\documentclass[12pt]{amsart}

\usepackage{enumerate, amsmath, amsthm, amsfonts, amssymb, xy,  mathrsfs, graphicx, paralist, lscape, multicol}

\usepackage[usenames, dvipsnames]{color}
\usepackage[margin=1in]{geometry} 
\usepackage[bookmarks, bookmarksdepth=2, colorlinks=true, urlcolor=blue, linkcolor=blue, citecolor=blue]{hyperref}
\usepackage[all]{hypcap}

\input xy
\xyoption{all}

\allowdisplaybreaks

\numberwithin{equation}{section}
\newtheorem{theorem}{Theorem}
\numberwithin{theorem}{section}

\newtheorem{proposition}[theorem]{Proposition}
\newtheorem{lemma}[theorem]{Lemma}
\newtheorem{corollary}[theorem]{Corollary}
\newtheorem{conjecture}[theorem]{Conjecture}

\theoremstyle{definition}
\newtheorem{rmk}[theorem]{Remark}
\newenvironment{remark}[1][]{\begin{rmk}[#1] \pushQED{\qed}}{\popQED \end{rmk}}
\newtheorem{eg}[theorem]{Example}
\newenvironment{example}[1][]{\begin{eg}[#1] \pushQED{\qed}}{\popQED \end{eg}}
\newtheorem{defn}[theorem]{Definition}

\newcommand{\rB}{\mathrm{B}}

\newcommand{\rC}{\mathrm{C}}

\newcommand{\rD}{\mathrm{D}}

\newcommand{\cE}{\mathcal{E}}

\newcommand{\rE}{\mathrm{E}}

\newcommand{\bF}{\mathbf{F}}

\newcommand{\rG}{\mathrm{G}}

\newcommand{\rH}{\mathrm{H}}

\newcommand{\bL}{\mathbf{L}}
\newcommand{\cL}{\mathcal{L}}

\newcommand{\cM}{\mathcal{M}}

\newcommand{\bO}{\mathbf{O}}
\newcommand{\cO}{\mathcal{O}}

\newcommand{\bP}{\mathbf{P}}

\newcommand{\cR}{\mathcal{R}}

\newcommand{\rR}{\mathrm{R}}

\newcommand{\bS}{\mathbf{S}}
\newcommand{\cS}{\mathcal{S}}

\newcommand{\cT}{\mathcal{T}}

\newcommand{\cV}{\mathcal{V}}

\newcommand{\bZ}{\mathbf{Z}}

\newcommand{\fu}{\mathfrak{u}}

\renewcommand{\phi}{\varphi}
\renewcommand{\emptyset}{\varnothing}
\newcommand{\eps}{\varepsilon}

\renewcommand{\tilde}[1]{\widetilde{#1}}

\newcommand{\arxiv}[1]{\href{http://arxiv.org/abs/#1}{{\tt arXiv:#1}}}

\makeatletter
\def\Ddots{\mathinner{\mkern1mu\raise\p@
\vbox{\kern7\p@\hbox{.}}\mkern2mu
\raise4\p@\hbox{.}\mkern2mu\raise7\p@\hbox{.}\mkern1mu}}
\makeatother

\DeclareMathOperator{\codim}{codim}

\renewcommand{\hom}{\operatorname{Hom}}

\DeclareMathOperator{\rank}{rank}

\DeclareMathOperator{\Pf}{Pf}
\DeclareMathOperator{\Sym}{Sym}

\DeclareMathOperator{\Tor}{Tor}

\DeclareMathOperator{\Spec}{Spec}

\DeclareMathOperator{\Pic}{Pic}
\DeclareMathOperator{\gr}{gr}

\newcommand{\GL}{\mathbf{GL}}

\newcommand{\Sp}{\mathbf{Sp}}
\newcommand{\SO}{\mathbf{SO}}
\newcommand{\Gr}{\mathbf{Gr}}
\newcommand{\Fl}{\mathbf{Fl}}
\newcommand{\IGr}{\mathbf{IGr}}
\newcommand{\IFl}{\mathbf{IFl}}
\newcommand{\OGr}{\mathbf{OGr}}

\newcommand{\Spin}{\mathbf{Spin}}

\title{Littlewood complexes and analogues of determinantal varieties}
\author{Steven V Sam}
\address{Department of Mathematics, University of California, Berkeley, CA}
\email{svs@math.berkeley.edu}
\author{Jerzy Weyman}
\address{Department of Mathematics, University of Connecticut, Storrs, CT}
\email{jerzy.weyman@uconn.edu}
\date{May 22, 2014}

\begin{document}

\maketitle

\begin{abstract}
One interesting combinatorial feature of classical determinantal varieties is that the character of their coordinate rings give a natural truncation of the Cauchy identity in the theory of symmetric functions. Natural generalizations of these varieties exist and have been studied for the other classical groups. In this paper we develop the relevant properties from scratch. By studying the isotypic decomposition of their minimal free resolutions one can recover classical identities due to Littlewood for expressing an irreducible character of a classical group in terms of Schur functions. We propose generalizations for the exceptional groups. In type $\rG_2$, we completely analyze the variety and its minimal free resolution and get an analogue of Littlewood's identities. We have partial results for the other cases. In particular, these varieties are always normal with rational singularities.
\end{abstract}

\tableofcontents

\section{Introduction.}

D.~E. Littlewood investigated the construction of irreducible representations of the orthogonal, symplectic and symmetric groups by means of traceless tensors \cite{littlewood}. The combinatorial aspects of the first two cases of this construction have been well understood in the subsequent work of Koike--Terada \cite{koike} and Sundaram \cite{sundaram}. More recently, deeper results were obtained in joint work of the authors with Snowden \cite{ssw}. 

The purpose of this paper is to generalize the results of Littlewood
for the symplectic and orthogonal group to the exceptional Lie
groups. Our approach is geometric and allows us to treat all of the cases in a uniform way. We define and study analogues of determinantal varieties, which we call {\bf Littlewood varieties}, for the classical and exceptional groups. (See \cite{dEP} for a combinatorial treatment of ``type A'' determinantal varieties.) The overall idea is as follows. Simple groups (except for the group of type ${\rm E}_8$) have small representations, i.e., non-trivial
irreducible representations having dimension smaller than that of the
adjoint representation. For the classical groups these are the vector
representations, and for groups of type ${\rm G}_2$, ${\rm F}_4$,
${\rm E}_6$, and ${\rm E}_7$ these are the fundamental representations
$V(\omega_1)$, $V(\omega_4)$, $V(\omega_6)$, and $V(\omega_7)$,
respectively, in Bourbaki notation. (For $\rE_8$, we call its adjoint representation the small representation.) For a small representation $V$ it
is natural to ask whether other irreducible representations can be
constructed in a natural way from tensor powers of $V$. It is
well-known that this can be done for the classical groups (via Young's
construction of Schur functors and Weyl's construction via traceless tensors for classical groups).

There is also a related question: Can one write the character of an
irreducible representation as an alternating sum of characters of
Schur functors applied to $V$? An even more precise result would be to
construct an acyclic complex whose terms are direct sums of Schur
functors on $V$ and whose cokernel is the irreducible representation
$V_\lambda$ of highest weight $\lambda$, or just to construct a
presentation of $V_\lambda$ by a map of Schur functors on $V$. 

For classical groups such character formulas for irreducible
representations $V_\lambda$ were given by Littlewood
\cite{littlewood}.  In \cite[Chapter 6, exercises]{weyman} a procedure
for constructing the corresponding complex $C(\lambda)_\bullet$ is
given. They are obtained as the isotypic components of the Koszul
complex giving the syzygies of a certain complete intersection $Y$. We consider these for general Littlewood varieties, and we call the complexes $C(\lambda)_\bullet$ {\bf Littlewood complexes}.

In the present paper we give the analogues of these constructions for
the exceptional groups. Carr and Garibaldi \cite{carr} gave interpretations of the homogeneous spaces for exceptional groups in terms of flags in a small
representation $V$ (for the type ${\rm G}_2$ it was also done by
Anderson \cite{anderson}). Before that, some fundamental work on the
small representation of ${\rm E}_6$ was done by Aschbacher
\cite{aschbacher}. For each subspace $W_i$ in these flags we proceed
as follows. Let $\cR_i$ be the corresponding tautological subbundle on
the appropriate homogeneous space $G/P$. Consider several copies of
$V$, i.e., the representation $\hom_K (E, V)$ for some vector space $E$. We construct a variety $Y\subset \hom_K(E, V)$ which is desingularized by the vector bundle $\hom(E, \cR_i)$ over $G/P$. We use the geometric technique for calculating syzygies \cite[Chapter 5]{weyman} to draw the consequences. We work most of the time under the assumption $\dim E=i=\dim W_i$. One
could do the same in the case $\dim E > \dim W_i$ by adding the
rank conditions to the equations defining $Y$.

We prove that the varieties $Y$ are normal with rational singularities. 
In small cases we construct the syzygies of the coordinate rings of $Y$. In some cases (we call them the spherical cases) the coordinate rings
of the subvarieties $Y\subset \hom_K (E, V)$ are multiplicity-free.
Then they have the decomposition
\begin{align} \label{eqn:cauchy}
  K[Y]=\bigoplus_\lambda \bS_\lambda E \otimes V_{[\lambda]},
\end{align}
where $V_{[\lambda]}$ is an irreducible representation of the group
$G$ with highest weight $[\lambda]$. The notation $[\lambda]$ is explained in each individual section, but for now we note that it has the property 
$[\lambda] + [\mu] = [\lambda + \mu]$. As $\lambda$ varies through all
the partitions with $\ell(\lambda) \le \dim E$, the weight $[\lambda]$
varies through some sublattice $\Lambda_G(V)^+\subseteq \Lambda_G^+$
of the weight lattice of $G$. In some of the spherical cases (in particular, the case of three
copies of ${\rm F}_4$ and ${\rm E}_6$ and the case of four copies of
${\rm E}_7$) we prove that the defining ideals of the varieties $Y$
are generated by quadrics giving presentations of the representations
$V_{[\lambda]}$ by means of Schur functors on $V$.

One can view such a formula as an analogue of the Cauchy formula for
the decomposition of the coordinate ring of a determinantal variety
into the irreducible representations of a product of general linear
groups (see \cite[(6.1.5)(d)]{weyman}). For a variety $Y$ with the coordinate ring $K[Y]$ as in \eqref{eqn:cauchy} the isotypic components of its minimal free resolution (over the polynomial ring) give us the desired Littlewood complexes $C(\lambda)_\bullet$. Part of our motivation to
construct varieties $Y$ with such coordinate rings stems from the work
of the first author on saturation theorems for the classical groups
\cite{sam} (building on the work \cite{derksenweyman}) and to find analogous situations involving the exceptional groups. 

The paper is organized as follows. The sections describe the varieties
in question type by type, so we cover in succession: types ${\rm C}_n$, ${\rm B}_n$, ${\rm D}_n$, ${\rm G}_2$, ${\rm F}_4$, ${\rm E}_6$, ${\rm E}_7$, and ${\rm E}_8$. In types $\rB$ and $\rD$, these varieties do not see the spin representations. These can be accounted for by considering a ``spinor module'' which is supported on the corresponding variety, and we study this in \S\ref{sec:lwoodspin}.

\begin{remark} 
We mention that Brion and Inamdar have shown that spherical varieties are Frobenius split in sufficiently large characteristics, and hence have rational singularities \cite{brion}. It would be interesting to see if one can use Frobenius splitting techniques to obtain quadratic generation of the ideals mentioned above.
\end{remark}

\subsection*{Notation and conventions.}

We use Bourbaki's numbering for the nodes of Dynkin diagrams. This is consistent with the labeling in LiE \cite{lie}. If $\lambda = (\lambda_1, \lambda_2, \dots)$ is a partition (i.e., $\lambda_i \ge \lambda_{i+1}$ and $\lambda_i = 0$ for $i \gg 0$), then $\lambda^\dagger$ denotes the transposed partition. In symbols, $\lambda^\dagger_i = \#\{j \mid \lambda_j \le i\}$. The rank of $\lambda$, denoted $\rank(\lambda)$, is the length of the main diagonal of the associated Young diagram, i.e., $\max\{i \mid \lambda_i \ge i\}$. Finally, $|\lambda| = \sum_i \lambda_i$ and $\ell(\lambda) = \max\{r \mid \lambda_r \ne 0\}$.

\subsection*{Acknowledgements.} We thank Skip Garibaldi for making us
aware of his papers. We also thank Jason Ribeiro for providing
software to use Bott's algorithm and Witold Kra\'skiewicz and Evan
Klitzke for assistance with some of the larger computer
calculations. The software packages LiE \cite{lie} and Macaulay2
\cite{M2} were crucial to many parts of this work. 

Steven Sam was supported by an NDSEG fellowship and a Miller research fellowship. Jerzy Weyman was partially supported by NSF grant DMS-0901185.

\section{Preliminaries.}

Except for \S\ref{sec:BWB}, $K$ is an arbitrary field. We use $G$ to denote a connected split reductive group. We let $B \subset G$ denote a Borel subgroup and let $U$ be the unipotent radical of $B$. We will usually assume that $G$ is simply-connected.

\subsection{Geometric technique.} \label{sec:geomtech}

For more details, we refer to \cite[Chapter 5]{weyman}.

Let $X$ be a vector space of dimension $N$, which we identify with $\Spec(A)$ where $A = \Sym(X^*)$ is a graded ring. Let $V$ be a smooth projective variety and let $\cE = X \times V$ be a trivial vector bundle of rank $N$, together with the projection maps $q \colon \cE \to X$ and $p \colon \cE \to V$. Consider a short exact sequence of vector bundles
\[
0 \to \cS \to \cE \to \cT \to 0
\]
and let $Z$ be the total space of $\cS$. We define the variety $Y
\subset X$ to be the image $q(Z)$ with reduced subscheme structure. We
also set $\eta = \cS^*$ and $\xi = \cT^*$.

Let $\cV$ be a vector bundle on $V$. For each $i \in \bZ$, define the $A$-module 
\[
\bF(\cV)_i = \bigoplus_{j \ge 0} \rH^j(V; (\bigwedge^{i+j} \xi) \otimes \cV) \otimes_K
A(-i-j). 
\]

\begin{theorem} \label{thm:geom-tech}
There exist minimal differentials $d_i \colon \bF(\cV)_i \to \bF(\cV)_{i-1}$ of degree $0$ such that
\[
\rH_{-i}(\bF(\cV)_\bullet) = \rR^i q_* (\cO_Z \otimes p^* \cV) = \rH^i(V; \Sym(\eta) \otimes \cV).
\]
for all $i \ge 0$ and $\rH_j(\bF(\cV)_\bullet) = 0$ for all $j > 0$.
\end{theorem}

We will particularly be interested in the case when the higher direct images ($i>0$) vanish, and $q$ is birational. In this case, when $\cV = \cO_V$, we can identify $q_*\cO_Z$ with the normalization of $\cO_Y$ and $\bF_\bullet$ gives an $A$-linear minimal free resolution of $q_*\cO_Z$.

Finally, we state the criterion from \cite[Corollary 5.1.5]{weyman} for the module 
\[
M(\cV) = \rH^0(V; \Sym(\eta) \otimes \cV)
\]
to be Cohen--Macaulay. Let $\omega_V$ be the canonical bundle of $V$ and define 
\[
\cV^\vee = \omega_V \otimes \det \xi^* \otimes \cV^*.
\]

\begin{theorem} \label{thm:CM-geom}
$M(\cV)$ is Cohen--Macaulay if $\rH^i(V; \Sym(\eta) \otimes \cV) = \rH^i(V; \Sym(\eta) \otimes \cV^\vee) = 0$ for all $i>0$.
\end{theorem}

\subsection{Schur modules and good filtrations.} \label{sec:schurmodules}

The Picard group of the flag variety $G/B$ is identified with the weight lattice $\Lambda_G$ of $G$. Given a weight $\lambda$, we let $\cL_\lambda$ denote the corresponding line bundle, and we index this so that if $\lambda$ is dominant, then $\cL_\lambda$ is generated by its global sections. Furthermore, we have that $\cL_\lambda \otimes \cL_\mu = \cL_{\lambda + \mu}$.

Given a dominant weight $\lambda$, we know by Kempf's vanishing
theorem \cite[Theorem 3.1.1(a)]{frobenius} that the higher cohomology
of $\cL_\lambda$ vanishes. The space of sections 
\[
V_\lambda := \rH^0(G/B; \cL_\lambda)
\]
is a highest weight module which we call a {\bf Schur module}. A finite-dimensional representation $W$ has a {\bf good filtration} if $W$ has a filtration whose associated graded $\gr(W)$ is a direct sum of Schur modules. We will use the notation
\[
W \approx \gr(W)
\]
to denote this. A similar definition is made for graded representations with finite-dimensional components. This condition is automatic when $K$ has characteristic $0$ since then all finite-dimensional representations of $G$ are semisimple and the $V_\lambda$ exhaust all of the simple representations. Given two vector spaces $E$ and $F$, the symmetric algebra $\Sym(E \otimes F)$ has a good filtration with respect to the action of $\GL(E) \times \GL(F)$ \cite[Theorem 2.3.3(a)]{weyman}:
\begin{align} \label{eqn:cauchy-filt}
\Sym(E \otimes F) \approx \bigoplus_\lambda \bS_\lambda(E) \otimes \bS_\lambda(F).
\end{align}

For a partition $\lambda$, the notation $\bS_\lambda E$ denotes the Schur functor applied to $E$, see \cite[Chapter 2]{weyman}. Note that $\bS_\lambda E = 0$ if $\ell(\lambda)> \dim E$. But we change the indexing so that our
$\bS_\lambda E$ is denoted $\bL_{\lambda'} E$ in (loc. cit.). Given an inclusion of partitions $\mu \subseteq \lambda$ (i.e., $\mu_i \le \lambda_i$ for all $i$) one can also define the skew Schur functors $\bS_{\lambda/\mu} E$.

\subsection{Borel--Weil--Bott theorem.} \label{sec:BWB}

For this section, we assume that $K$ is of characteristic $0$. See \cite[Chapter 4]{weyman} or \cite{demazure} for details. We also point the reader to \cite[\S 12.1]{humphreys} as a reference for coordinate systems for root systems. In particular, for the classical types ABCD, the $\eps_i$ basis is what we call ``partition notation''.

We have a shifted Weyl group action on the set $\Lambda_G$ of integral weights given by
\[
w\bullet\lambda:=w(\lambda+\rho)-\rho
\]
where $\rho$ is half of the sum of positive roots. We use $V_\lambda$ to denote the irreducible rational $G$-module with highest weight $\lambda$. The following theorem calculates the cohomology of line bundles on $G/B$.

\begin{theorem}[Bott] Let $\cL_\lambda$ be a line bundle on $G/B$ corresponding to a weight $\lambda$. We have two mutually exclusive possibilities.
\begin{compactenum}[\rm 1.]
\item There exists a non-trivial element $w\in W$ such that $w\bullet\lambda=\lambda$. Then all cohomology groups of $\cL_\lambda$ are zero.
\item There exists a unique $w\in W$ such that $\mu:=w\bullet\lambda$ is dominant.
Then the only non-zero cohomology group of $\cL_\lambda$ is
\[
\rH^{\ell(w)}(G/B, \cL_\lambda)=V_\mu
\]
where $\ell(w)$ denotes the length of $w$.
\end{compactenum}
\end{theorem}

Alternatively, instead of using a single Weyl group element $w \in W$, we could apply the shifted action of simple reflections (such that our weight becomes closer to dominant at each step) and keep track of how many we use. We will refer to this as ``Bott's algorithm''.

There are analogous results over other homogeneous spaces $G/P$ for parabolic subgroups $P$. They are obtained from the above result by applying the Leray spectral sequence to the composition $G/B\rightarrow G/P\rightarrow *$ along with a relative version of Bott's theorem. We will use this later, but explain it on a case-by-case basis to simplify notation.

\subsection{Graded $G$-algebras.} \label{sec:gradedGalg}

Given two dominant weights $\lambda$ and $\mu$, we get a multiplication map on sections
\[
V_\lambda \otimes V_\mu = \rH^0(G/B; \cL_\lambda) \otimes \rH^0(G/B;
\cL_\mu) \to \rH^0(G/B; \cL_{\lambda + \mu}) = V_{\lambda + \mu}.
\]
This is the {\bf Cartan product}, and is surjective \cite[Proposition 1.5.1]{frobenius}.

Let $\lambda^1, \dots, \lambda^N$ be dominant weights. Using the Cartan product, we can define a $G$-equivariant multi-graded algebra
\[
\bigoplus_{d_1, \dots, d_N \ge 0} V_{d_1 \lambda^1 + \cdots + d_N \lambda^N}
\]
which is a quotient of the multi-graded symmetric algebra $\Sym(V_{\lambda^1}) \otimes \cdots \otimes \Sym(V_{\lambda^N})$. Then the ideal of this quotient is generated by elements of total degree 2, and in fact these algebras are Koszul \cite[Theorem 3]{inamdar}.

\subsection{Grosshan's deformation.} \label{sec:grosshans}

See \cite[\S 15]{grosshans} for more details on the material in this section. Grosshans constructs a homomorphism $h \colon \Lambda_G\rightarrow\bZ$ satisfying the conditions
\begin{compactitem}
\item $h(\omega) \ge 0$ for $\omega\in\Lambda_G^+$,
\item If $\chi >\chi'$ (i.e., $\chi-\chi'$ is a sum of positive roots), then $h(\chi) > h(\chi')$.
\end{compactitem}
Let $A$ be a commutative $K$-algebra with a rational $G$-action. Set
\begin{align*}
A_n &=\lbrace a\in A\ |\ h(\chi)\le n \text{ for all weights } \chi \text{ of } T  \text{ in the span }\langle G \cdot a\rangle\rbrace\\
\gr(A) &=\bigoplus_{n\ge 0} A_n/A_{n-1}.
\end{align*}
Then $\gr(A)$ is an algebra with a rational $G$-action, with a product induced from the product on $A$. The algebras $A$ and $\gr(A)$ have the same subring of $U$-invariants. Define 
\[
D=\sum_{n\ge 0} A_n x^n\subset A[x].
\]

The algebra $D$ has a rational $G$-action and has the following properties.

\begin{theorem} Let $i \colon K[x]\rightarrow D$ be the inclusion. Then $D$ is flat over $K[x]$ and the fiber of $i$ over the ideal $(x-\alpha)$ for $\alpha\ne 0$ is isomorphic to $A$, and the fiber over $(x)$ is isomorphic to $\gr(A)$.
\end{theorem}

One example of the algebra $\gr(A)$ is the multi-graded algebra with the Cartan product considered in \S\ref{sec:gradedGalg}, so one has the following application. For a proof, see \cite[Theorem 6.2]{landsbergweyman}.

\begin{theorem} 
Let $K$ be a field of characteristic $0$. Assume that the algebra $A$ is an integral domain and is multiplicity-free as a $G$-module. If $\gr(A)$ is generated in degrees $\le d$ then the defining ideal of $A$ is generated in degrees $\le 2d$.
\end{theorem}

\section{Classical groups.}

As a warmup, we review some known results about the varieties that we
study in the case of classical groups. As far as we are aware, not all
of these statements can be readily found in the literature. So we
prove the statements in the case of the symplectic group and just
state the relevant differences for the orthogonal group (the proofs
proceed in the same way).

We use the language of partitions for the weights. For the symplectic case, this means that we have a basis $v_1, \dots, v_{2n}$ for $V$ chosen in such a way that the symplectic form $\omega$ is defined by $\omega(v_i, v_j) = 0$ if $|i-j| \ne n$ and $\omega(v_i, v_{i+n}) = 1$ for $i=1,\dots,n$. Then our maximal torus consists of diagonal matrices of the form ${\rm diag}(t_1, \dots, t_n, t_1^{-1}, \dots, t_n^{-1})$ and the weight $\lambda =(\lambda_1,\ldots ,\lambda_n)$ is defined by $(t_1, \dots, t_n) \mapsto t_1^{\lambda_1} \cdots t_n^{\lambda_n}$.
The fundamental weights are $\omega_i = (1^i, 0^{n-i})$ for $1\le i\le n$. The orthogonal case is similar, and more details are given in \S\ref{sec:typeBn}.

The symplectic and orthogonal Schur functors $\bS_{[\lambda]} V$ are defined to be $V_\lambda$, and are quotients of $\bS_\lambda V$. They can also be defined intrinsically. More precisely, the Weyl construction (see \cite[\S \S 17.3, 19.5]{fultonharris} for the case of characteristic $0$) specifies that the symplectic Schur functor $\bS_{[\lambda]} V$
can be defined as the cokernel
\[
\bS_{\lambda /(1^2)} V\xrightarrow{\psi_\lambda} \bS_{\lambda} V\rightarrow \bS_{[\lambda]} V\rightarrow 0
\]
where the map $\psi_\lambda$ is the composition of the map
$\bS_{\lambda/(1^2)}V\rightarrow \bS_{\lambda/(1^2)}V\otimes\bigwedge^2 V$ given by the trace of the symplectic form on $V$ with a $\GL(V)$-equivariant map
\[
\bS_{\lambda /(1^2)} V\otimes \bigwedge^2 V\rightarrow \bS_{\lambda}V
\]
which under the adjoint property of the skew Schur functor corresponds to the map
\[
\bS_\lambda V\rightarrow \bS_\lambda V\otimes\bigwedge^2 V\otimes\bigwedge^2 V^*
\]
which is the multiplication by the trace of the identity. The simplest example is when $\lambda = (1^i)$ for some $i \ge 2$, in which case $\bS_{(1^i)} V = \bigwedge^i V$ and we have the presentation
\[
\bigwedge^{i-2} V \to \bigwedge^i V \to \bS_{[1^i]} V \to 0.
\]
Here the map is just exterior multiplication by the symplectic form $\omega \in \bigwedge^2 V \cong \bigwedge^2 V^*$.

For the orthogonal groups we use the partition $(2)$ instead of $(1^2)$, but the description is similar.


\subsection{Type ${\rm C}_n$.} \label{sec:typeCn}

In this section we recall the construction for the type ${\rm C}_n$. It serves as a model for the constructions for the exceptional groups. The construction of the complexes $C(\lambda)_\bullet$ is sketched in \cite[Exercises 6.1--6.4]{weyman}. Unfortunately these exercises contain mistakes, so we do everything from scratch here.

Let $V$ be a vector space of dimension $2n$ over a field $K$ of arbitrary characteristic equipped with a symplectic form $\omega$. We denote by $\bS_{[\lambda]} V^* \cong \bS_{[\lambda]} V$ the irreducible representation of $\Sp(V)$ of highest weight $\lambda$. In this case $\lambda$ is a partition with at most $n$ parts.

Pick $r,m$ such that $0 \le r \le m \le n$. Let $E$ be a vector space of dimension $m$. We denote by $Y_{r,m,n} \subset \hom(E,V)$ the subvariety of linear maps for which the image of $E$ is contained in an isotropic subspace of dimension $r$. When $r=m=n$, we call $Y$ the Littlewood variety. Also, let $A = K[\hom(E,V)]$. Fix a basis $e_1, \dots, e_m$ for $E$, so that we may represent an element of $\hom(E,V)$ with an $m$-tuple of vectors in $V$, which we write as $({\bf v}_1, \dots, {\bf v}_m)$. For $1 \le i < j \le m$, we define the function $x_{i,j}$ via $({\bf v}_1, \dots, {\bf v}_m) \mapsto \omega({\bf v}_i, {\bf v}_j)$.

The isotropic Grassmannian $\IGr(r,V)$ consists of all $r$-dimensional
subspaces of $V$ on which $\omega$ is identically $0$, so is naturally a subvariety of the usual Grassmannian $\Gr(r,V)$. Let $\cR_r$ denote the restriction of the tautological rank $r$ subbundle of $\Gr(r,V)$ to $\IGr(r,V)$. Then $\IGr(r,V)$ is smooth and connected, and is the zero locus of a regular section of $\bigwedge^2 \cR_r^*$ (see for example \cite[Proposition 4.3.6]{weyman}), so
\[
\dim \IGr(r,V) = \dim \Gr(r,V) - \binom{r}{2} = \frac{r(4n - 3r + 1)}{2}.
\]

The full flag variety for $\Sp(V)$ can be described as the space of flags of subspaces $V_1 \subset V_2 \subset \cdots \subset V_n \subset V$ where each $V_i$ is an isotropic subspace. Denote this by $\IFl(V)$. In particular, we have an obvious projection $\pi \colon G/B \to \IGr(r,V)$, and the fiber of $\pi$ over $W \in \IGr(r,V)$ is the variety $\Fl(W) \times \IFl(W^\perp / W)$ ($W^\perp/W$ inherits a symplectic form from $V$). In particular, $\pi$ is the relative flag variety $\Fl(\cR_r) \times \IFl(\cR^\perp_r / \cR_r)$. 

If $\lambda$ is a partition, then it is a dominant weight and we get a line bundle $\cL(\lambda)$ on $\IFl(V)$ as in \S\ref{sec:schurmodules}. If $\ell(\lambda) \le r$, then $\cL(\lambda)$ is fiberwise trivial on the $\IFl(\cR^\perp_r / \cR_r)$ component of $\pi$. In particular, $\pi_* \cL(\lambda) = \bS_\lambda(\cR_r^*)$ and the higher direct images vanish: $\rR^j \pi_* \cL(\lambda) = 0$ for $j > 0$. A consequence of the Leray spectral sequence \cite[5.8.6]{weibel} is then that 
\[
\rH^i(\IGr(r,V); \bS_{\lambda}(\cR_r^*)) = \rH^i(\IFl(V); \cL(\lambda)) \quad (i \ge 0).
\]
Combining this with Kempf vanishing for $\cL(\lambda)$ from \S\ref{sec:schurmodules}, we conclude that
\begin{align} \label{eqn:typeC-sheaf-coh}
\rH^i(\IGr(r,V); \bS_\lambda(\cR_r^*)) = \begin{cases} \bS_{[\lambda]}(V) & \text{if } i=0\\
0 & \text{if } i>0 \end{cases}.
\end{align}

\begin{theorem}
\begin{compactenum}[\rm (1)]
\item The variety $Y_{r,m,n}$ is irreducible and has codimension $(m-r)(2n-r) + \binom{r}{2}$. Furthermore, $Y_{r,m,n}$ has rational singularities, and in particular, is normal and Cohen--Macaulay.

\item As a $\GL(E) \times \Sp(V)$-module, the coordinate ring of $Y_{r,m,n}$ has a filtration with associated graded given by
\[
K[Y_{r,m,n}] \approx \bigoplus_{\lambda,\ \ell(\lambda) \le r} \bS_\lambda E \otimes \bS_{[\lambda]} V^*.
\]

\item The ideal defining $Y_{r,m,n}$ is minimally generated by the functions $x_{i,j}$ for $1 \le i < j \le m$ together with the minors of order $r+1$ of the matrix $\phi$. In particular, $Y_{m,m,n}$ is a complete intersection in $\hom(E,V)$.

\item If $0 < r < m$, the singular locus of $Y_{r,m,n}$ is $Y_{r-1,m,n}$. The singular locus of $Y_{m,m,n}$ is $Y_{m-2,m,n}$.
\end{compactenum}
\end{theorem}

\begin{proof}
The variety $Y_{r,m,n}$ has a desingularization which is the vector bundle $\hom(E, \cR_r) = E^* \otimes \cR_r$ over $\IGr(r,V)$. This desingularization shows that $Y_{r,m,n}$ has codimension $(m-r)(2n-r) + \binom{r}{2}$ and is irreducible. In the notation of \S\ref{sec:geomtech}, we have $\eta = E \otimes \cR_r^*$. Write $A_\lambda = \bS_\lambda E \otimes \bS_{[\lambda]} V^*$. We have
\[
K[\tilde{Y}_{r,m,n}] = \rH^0 (\IGr(r,V); \Sym(E \otimes \cR^*_r)) \approx \rH^0 (\IGr(r,V); \bigoplus_{\lambda,\ \ell(\lambda) \le r} \bS_\lambda E\otimes \bS_\lambda \cR^*_r) = \bigoplus_{\lambda,\ \ell(\lambda) \le r} A_\lambda
\]
(the first equality follows from \S\ref{sec:geomtech}, the middle $\approx$ is \eqref{eqn:cauchy-filt}, and the last equality follows from \eqref{eqn:typeC-sheaf-coh}). The multiplication $A_\lambda \otimes A_\mu \to A_{\lambda + \mu}$ on this algebra is given by the Cartan product, so by \S\ref{sec:gradedGalg}, it is generated by $A_{(1^i)}$ for $i=0,1,\dots,r$. To prove that $Y_{r,m,n}$ is normal, it is enough to know that $A_{(1^i)}$ is in the $A$-submodule of $K[Y_{r,m,n}]$ generated by 1. But this follows since these functions are nonzero on matrices of rank at least $i$. 

By \eqref{eqn:cauchy-filt}, we have $\Sym(E \otimes \cR_r^*) \approx \bigoplus_\lambda \bS_\lambda E \otimes \bS_\lambda(\cR_r^*)$. Furthermore, the higher cohomology of the latter vector bundle vanishes by \eqref{eqn:typeC-sheaf-coh}, so $Y_{r,m,n}$ has rational singularities.

Let $I_{r,m,n}$ be the ideal generated by $x_{i,j}$ and the minors of order $r+1$ of the matrix $\phi$. It is clear that $I_{r,m,n}$ set-theoretically defines $Y_{r,m,n}$, so we need to show that $I_{r,m,n}$ is radical.

We first handle $m=r$. The scheme defined by $I_{m,m,n}$ is a complete intersection in $\hom(E,V)$, and hence defines a Cohen--Macaulay scheme. So to show that $I_{m,m,n}$ is radical, we just need to show that this scheme has a nonsingular point, which can be checked using the Jacobian criterion \cite[Theorem 18.15]{eisenbud}. To do this, pick a basis $v_1, \dots, v_{2n}$ for $V$ in such a way that the symplectic form $\omega$ is defined by $\omega(v_i, v_j) = 0$ if $|i-j| \ne n$ and $\omega(v_i, v_{i+n}) = 1$ for $i=1,\dots,n$. Let $u_{i,j}$ be the function on $\hom(E,V)$ that picks out the coefficient of $v_i$ in the image of $e_j$. Then we can write
\[
x_{i,j} = \sum_{k=1}^n (u_{k,i} u_{n+k,j} - u_{n+k,i} u_{k,j}).
\]
Now let $\phi \in \hom(E,V)$ be the matrix defined by $e_i \mapsto v_i$ for $i=1,\dots,m$. We will show that $\phi$ is a nonsingular point via the Jacobian criterion. Note that $u_{i,j}(\phi) = \delta_{i,j}$. In particular $\frac{\partial x_{i,j}}{\partial u_{k,\ell}}(\phi) = 0$ unless $(k,\ell) = (n+i,j)$ or $(k,\ell) = (n+j,i)$, in which case it is 1 and $-1$, respectively. In particular, for any $(k,\ell)$, we see that $\frac{\partial x_{i,j}}{\partial u_{k,\ell}}(\phi)$ is nonzero for at most 1 value of $(i,j)$, which shows that the Jacobian has full rank, and hence $\phi$ is a nonsingular point and $I_{m,m,n}$ is radical. Note also that if we pick $\phi'$ to agree with $\phi$ except that $e_m \mapsto 0$, the Jacobian still has full rank at $\phi'$. The only difference is that $x_{i,m}$ will have only one nonzero partial derivative at $\phi'$ instead of two. But note that if we also had $e_{m-1} \mapsto 0$, then all partial derivatives of $x_{i,m}$ would be 0. This shows that the singular locus of $Y_{m,m,n}$ is $Y_{m-2,m,n}$.

Hence we have shown that 
\[
A/I_{m,m,n} \approx \bigoplus_{\lambda} A_\lambda.
\]
Using the Cartan product from \S\ref{sec:gradedGalg}, it is clear that after modding out by all minors of size $r+1$ of the matrix $\phi$, we get
\[
A/I_{r,m,n} \approx \bigoplus_{\lambda,\ \ell(\lambda) \le r} A_\lambda.
\]
But we have already seen that this is the coordinate ring of $Y_{r,m,n}$, so $I_{r,m,n}$ is radical. Finally, we apply the Jacobian criterion to $I_{r,m,n}$. We discussed $r=m$ above, so assume $r<m$. Let $\psi$ be the matrix $e_i \mapsto v_i$ for $i=1,\dots,r-1$. Then the partial derivatives of all minors of order $r$ are 0 at $\psi$. Also, $x_{i,j}$ has a nonzero partial derivative at $\psi$ if and only if $i \le r-1$, so the rank of the Jacobian at $\psi$ is $(r-1)m - \binom{r}{2}$, which is less than $(m-r)(2n-r) + \binom{r}{2} = \codim Y_{r,m,n}$ since $m > r$.
\end{proof}

Since $Y_{m,m,n}$ is a complete intersection, its minimal free resolution $\bF_\bullet$ is the Koszul complex with $\bF_i = \bigwedge^i \bigwedge^2 E \otimes A(-2i)$. We will be most interested in the case $m=n$. For the remainder of this section, we work over a field of characteristic $0$ for simplicity. Let $Q_{-1}$ be the set of partitions $\lambda$ such that $\bS_\lambda E$ appears in $\bigwedge^i \bigwedge^2 E$ for some $i$. This set can be described recursively as follows \cite[I.A.7, Ex. 5]{macdonald}: The empty partition belongs to $Q_{-1}$. A non-empty partition $\mu$ belongs to $Q_{-1}$ if and only if the number of rows in $\mu$ is one more than the number of columns, i.e., $\ell(\mu) = \mu_1 + 1$, and the partition obtained by deleting the first row and column of $\mu$, i.e., $(\mu_2 - 1, \dots, \mu_{\ell(\mu)} - 1)$, belongs to $Q_{-1}$.

\begin{proposition} \label{prop:typeClwood}
Assume that $K$ has characteristic $0$. The $\bS_\lambda(E)$-isotypic component of the Koszul complex ${\bf F}_\bullet$ gives an $\Sp(V)$-equivariant resolution $0 \to C(\lambda)_\bullet \to \bS_{[\lambda]} V^* \to 0$ where
\[
C(\lambda)_i = \bigoplus_{\mu \vdash 2i,\ \mu \in Q_{-1}} \bS_{\lambda / \mu} V^*.
\]
\end{proposition}

\begin{proof}
The coordinate ring of $Y$ is 
\[
K[Y] = \bigoplus_\lambda \bS_\lambda E \otimes \bS_{[\lambda]} V^*
\]
where $\bS_{[\lambda]} V^*$ denotes the irreducible representation of
$\Sp(V)$ of highest weight $\lambda$. Let's take the $\bS_\lambda
E$-isotypic component of this resolution. Note that
\[
A = \bigoplus_\mu \bS_\mu E \otimes \bS_\mu V^*.
\]
So the $\bS_\lambda E$-isotypic component of $\bS_{\mu} E \otimes A(-2i)$ can be written as $\bS_{\lambda / \mu} V^*$ (see for example, \cite[Theorem 2.3.6]{weyman}). Hence we get the $\Sp(V)$-equivariant resolution $0 \to C(\lambda )_\bullet \to \bS_{[\lambda]} V^* \to 0$ where
\[
C(\lambda )_i = \bigoplus_{\mu \vdash 2i,\ \mu \in Q_{-1}} \bS_{\lambda /  \mu} V^*. \qedhere
\]
\end{proof}

\begin{example}
Take $n=4$ and $\dim E = 3$. The resolution is
\[
0 \to \bS_{2,2,2} E \otimes A(-6) \to \bS_{2,1,1} E \otimes A(-4) \to
\bS_{1,1}E \otimes A(-2) \to A \to K[Y] \to 0. \qedhere
\]
\end{example}
  
\begin{remark} 
When $m > n$, the homology of the Koszul complex on the invariants $x_{i,j}$ is calculated in \cite{ssw} (this was first done in \cite{enright} but relies on more involved machinery).
\end{remark}


\subsection{Type ${\rm B}_n$.} \label{sec:typeBn}

This will be similar to the case of type ${\rm C}_n$, so we just state
the differences.

Let $V$ be a vector space of dimension $2n+1$ over a field $K$
equipped with a nonsingular quadratic form $q$. Then $q$ naturally
defines an orthogonal form $\omega$ via
\[
\omega(x,y) = q(x+y) - q(x) - q(y).
\]
We remark that $q$ is determined by $\omega$ if ${\rm char}\ K \ne
2$. A subspace $U \subset V$ is {\bf isotropic} if $q(u) = 0$ for all
$u \in U$. In that case, we have $\dim U \le n$. 

We denote by $\bS_{[\lambda]} V^*$ the irreducible representation of
$\SO(V)$ of highest weight $\lambda$. In this case $\lambda$ is a
partition with at most $n$ parts.

Pick $r,m$ such that $0 \le r \le m \le n$. Let $E$ be a vector space
of dimension $m$. We denote by $Y_{r,m,n} \subset \hom(E,V)$ the
subvariety of linear maps for which the image of $E$ is contained in
an isotropic subspace of dimension $r$. When $r=m=n$, we call $Y$ the Littlewood variety. Also, let $A = K[\hom(E,V)]$. Fix a basis $e_1, \dots, e_m$ for $E$, so that we may represent an element of $\hom(E,V)$ with an $m$-tuple of vectors in $V$, which we write as $({\bf v}_1, \dots, {\bf v}_m)$. For $1 \le i < j \le m$, we
define the function $x_{i,j}$ via $({\bf v}_1, \dots, {\bf v}_m)
\mapsto \omega({\bf v}_i, {\bf v}_j)$. For $1 \le i \le m$, we define
the function $x_{i,i}$ via $({\bf v}_1, \dots, {\bf v}_m) \mapsto
q({\bf v}_i)$.

The orthogonal Grassmannian $\OGr(r,V)$ consists of all $r$-dimensional subspaces of $V$ on which $q$ is identically $0$, so is naturally a subvariety of the usual Grassmannian $\Gr(r,V)$. Let $\cR_r$ denote the restriction of the tautological rank $r$ subbundle of $\Gr(r,V)$ to $\OGr(r,V)$. Then $\OGr(r,V)$ is smooth, connected, and is the zero locus of a regular section of $\Sym^2(\cR_r^*)$ (see for example \cite[Proposition 4.3.8]{weyman}), so 
\[
\dim \OGr(r,V) = \dim \Gr(r,V) - \binom{r+1}{2} = \frac{r(4n - 3r - 3)}{2}.
\]

Just as in \S\ref{sec:typeCn}, we get the following calculation:
\begin{align}
\rH^i(\OGr(r,V); \bS_\lambda(\cR_r^*)) = \begin{cases} \bS_{[\lambda]}(V) & \text{if } i=0\\
0 & \text{if } i>0 \end{cases}.
\end{align}

\begin{theorem}
\begin{compactenum}[\rm (1)]
\item The variety $Y_{r,m,n}$ is irreducible and has codimension $(m-r)(2n+1-r) + \binom{r+1}{2}$. Furthermore, $Y_{r,m,n}$ has rational singularities, and in particular, is normal and Cohen--Macaulay.

\item As a $\GL(E) \times \SO(V)$-module, the coordinate ring of $Y_{r,m,n}$ has a filtration with associated graded given by
\[
K[Y_{r,m,n}] \approx \bigoplus_{\lambda,\ \ell(\lambda) \le r} \bS_\lambda E \otimes \bS_{[\lambda]} V^*.
\]

\item The ideal defining $Y_{r,m,n}$ is minimally generated by the functions $x_{i,j}$ for $1 \le i \le j \le n$, together with the minors of order $r+1$ of the matrix $\phi$. In particular, $Y_{m,m,n}$ is a complete intersection in $\hom(E,V)$.

\item If $r>0$, the singular locus of $Y_{r,m,n}$ is $Y_{r-1,m,n}$.
\end{compactenum}
\end{theorem}

Note that the singular locus of $Y_{m,m,n}$ is $Y_{m-1,m,n}$ instead of $Y_{m-2,m,n}$, in contrast with the symplectic case. This is because the functions $x_{i,j}$ are defined for $i=j$ as well as $i<j$.

The analogue of Proposition~\ref{prop:typeClwood} holds in this case with one change: the set $Q_{-1}$ is replaced with $Q_{1}$, and we have $\mu \in Q_{1}$ if and only if $\mu^\dagger \in Q_{-1}$. We state it for completeness. Again, we do not need the assumption on characteristic of $K$, but it simplifies the argument and statement.

\begin{proposition} \label{prop:lwood-typeB}
Assume that $K$ has characteristic $0$. The $\bS_\lambda(E)$-isotypic component of the Koszul complex ${\bf F}_\bullet$ gives an $\SO(V)$-equivariant resolution $0 \to C(\lambda)_\bullet \to \bS_{[\lambda]} V^* \to 0$ where
\[
C(\lambda)_i = \bigoplus_{\mu \vdash 2i,\ \mu \in Q_{1}} \bS_{\lambda / \mu} V^*.
\]
\end{proposition}

\begin{remark} 
When $m > n$, the homology of the Koszul complex on the invariants $x_{i,j}$ is calculated in \cite{ssw} (this was first done in \cite{enright} but relies on more involved machinery).
\end{remark}
  

\subsection{Type ${\rm D}_n$.} \label{sec:typeDn}

There are some important differences between the results in the case
${\rm D}_n$ and the cases ${\rm B}_n$ and ${\rm C}_n$, but the proofs
are still similar.

Let $V$ be a vector space of dimension $2n$ over a field $K$ equipped with a nonsingular quadratic form $q$, which gives rise to an orthogonal form $\omega$ as in \S\ref{sec:typeBn}. We denote by $\bS_{[\lambda]} V^*$ the irreducible representation of $\bO(V)$ of highest weight $\lambda$. In this case $\lambda$ is a partition with at most $n$ parts. When $\ell(\lambda) = n$, this is the direct sum of two irreducible representations of $\SO(V)$, which we call $\bS_{[\lambda]^+} V^*$ and $\bS_{[\lambda]^-} V^*$. Their highest weights are $(\lambda_1,\ldots, \lambda_{n-1}, \lambda_n)$ and $(\lambda_1,\ldots ,\lambda_{n-1}, -\lambda_n)$. For convenience, we will sometimes use $\bS_{[\lambda]^\pm} V^*$ to mean $\bS_{[\lambda]} V^*$ when $\ell(\lambda) < n$.

Pick $r,m$ such that $0 \le r \le m \le n$. Let $E$ be a vector space
of dimension $m$. We denote by $Y_{r,m,n} \subset \hom(E,V)$ the
subvariety of linear maps for which the image of $E$ is contained in
an isotropic subspace of dimension $r$. When $r=m=n$, we call $Y$ the Littlewood variety. Also, let $A = K[\hom(E,V)]$. Fix a basis $e_1, \dots, e_m$ for $E$, so that we may represent an
element of $\hom(E,V)$ with an $m$-tuple of vectors in $V$, which we
write as $({\bf v}_1, \dots, {\bf v}_m)$. For $1 \le i < j \le m$, we
define the function $x_{i,j}$ via $({\bf v}_1, \dots, {\bf v}_m)
\mapsto \omega({\bf v}_i, {\bf v}_j)$. For $1 \le i \le m$, we define
the function $x_{i,i}$ via $({\bf v}_1, \dots, {\bf v}_m) \mapsto
q({\bf v}_i)$.

The orthogonal Grassmannian $\OGr(r,V)$ consists of all $r$-dimensional subspaces of $V$ on which $q$ is identically $0$, so is naturally a subvariety of the usual Grassmannian $\Gr(r,V)$. Let $\cR_r$ denote the restriction of the tautological rank $r$ subbundle of $\Gr(r,V)$ to $\OGr(r,V)$. When $r=n$, $\OGr(n,V)$ has 2 connected components, which we denote $\OGr^+(n,V)$ and $\OGr^-(n,V)$. They are called {\bf spinor varieties}. Both components are isomorphic to one another. The variety $\OGr(r,V)$ is smooth and is the zero locus of a regular section of $\Sym^2(\cR_r^*)$ (see for example \cite[Proposition 4.3.8]{weyman}), so
\[
\dim \OGr(r,V) = \dim \Gr(r,V) - \binom{r+1}{2} = \frac{r(4n - 3r + 1)}{2}.
\]

\begin{theorem}
\begin{compactenum}[\rm (1)]
\item The variety $Y_{r,m,n}$ has codimension $(m-r)(2n-r) + \binom{r+1}{2}$. When $r < n$, $Y_{r,m,n}$ is irreducible. When $r=n$, $Y_{n,n,n}$ has $2$ irreducible components $Y^+_{n,n,n}$ and $Y^-_{n,n,n}$, and $Y^+_{n,n,n} \cap Y^-_{n,n,n} = Y_{n-1,n,n}$ (scheme-theoretically). In all cases, the irreducible components of $Y_{r,m,n}$ have rational singularities, and in particular, are normal and Cohen--Macaulay.

\item As a $\GL(E) \times \bO(V)$-module, the coordinate ring of $Y_{r,m,n}$ has a filtration with associated graded given by
\[
K[Y_{r,m,n}] \approx \bigoplus_{\lambda,\ \ell(\lambda) \le r} \bS_\lambda E \otimes \bS_{[\lambda]} V^*.
\]
Furthermore, we have
\[
K[Y^+_{n,n,n}] \approx \bigoplus_{\lambda} \bS_\lambda E \otimes \bS_{[\lambda]^+} V^*, \quad K[Y^-_{n,n,n}] \approx \bigoplus_\lambda \bS_\lambda E \otimes \bS_{[\lambda]^-} V^*.
\]

\item The ideal defining $Y_{r,m,n}$ is minimally generated by the
  functions $x_{i,j}$ for $1 \le i \le j \le m$ together with the
  minors of order $r+1$ of the matrix $\phi$. In particular, $Y_{m,m,n}$ is a complete
  intersection in $\hom(E,V)$. The additional equations needed to cut
  out $Y_{n,n,n}^+$ in $Y_{n,n,n}$ are the $\frac{1}{2} \binom{2n}{2}$
  $n \times n$ minors of $\phi$ whose column indices form a set
  $\{i_1, \dots, i_n\}$ such that $\#(\{i_1, \dots, i_n\} \cap \{1,
  \dots, n\})$ is even. A similar statement holds for $Y_{n,n,n}^-$
  except that even is replaced with odd.

\item If $r>0$, the singular locus of $Y_{r,m,n}$ is $Y_{r-1,m,n}$. The singular locus of $Y_{n,n,n}^\pm$ is $Y_{n-2,n,n}$.
\end{compactenum}
\end{theorem}

\begin{proof} 
The proof is similar to the last two cases, except a few points. 

First, the coordinate rings can be calculated using Theorem~\ref{thm:geom-tech} as we have done before. The proposed $n \times n$ minors generate the subspaces $\bigwedge^n E \otimes \bS_{[1^n]^+} V^*$ and $\bigwedge^n E \otimes \bS_{[1^n]^-} V^*$, respectively. Via the Cartan product, we see that the ideal generated by the first set of equations contains all $\bS_\lambda E \otimes \bS_{[\lambda]^+} V^*$ for $\ell(\lambda) = n$, and similarly for the second set. Hence these define radical ideals.

  The calculation of the singular locus of $Y_{r,m,n}$ proceeds as
  before, so we omit it. Now we consider $Y_{n,n,n}^\pm$. Fix a basis
  $\{v_1, \dots, v_{2n}\}$ for $V$ so that $q(\sum_{i=1}^{2n} c_iv_i)
  = \sum_{i=1}^n c_ic_{i+n}$.  Let $\psi$ be the matrix $e_i \mapsto
  v_i$ for $i=1,\dots,n-1$ and $e_n \mapsto 0$. Also let $u_{i,j}$ be
  the function on $\hom(E,V)$ that picks out the coefficient of $v_i$
  in the image of $e_j$.

  The functions $x_{i,j}$ with $1 \le i \le n-1$ and $i \le j$ give
  $\binom{n+1}{2} - 1$ linearly independent rows in the Jacobian
  matrix at $\psi$. Let $\Delta$ be the $n \times n$ minor with column
  indices $\{1, \dots, n\}$. Then $\frac{\partial \Delta}{\partial
    u_{n,n}}(\psi) = 1$, and $u_{n,n}$ is not involved in any of the
  nonzero partial derivatives of the functions $x_{i,j}$. Hence the
  Jacobian has maximal rank at $\psi$. On the other hand, for a rank
  $n-2$ matrix $\psi'$, the partial derivatives of all $n \times n$
  minors are 0 at $\psi'$, so the rank of the Jacobian drops. If $n$
  is even, this shows that the singular locus of $Y_{n,n,n}^+$ is
  $Y_{n-2,n,n}$. If $n$ is odd, we get the statement for
  $Y_{n,n,n}^-$. In either case, both components are exchanged under
  the action of an element in ${\bf O}(V) \setminus \SO(V)$, so we are
  done.
\end{proof}
  
A version of Proposition~\ref{prop:lwood-typeB} also holds. We should just make sure to use the full orthogonal group $\bO(V)$ instead of just the special orthogonal group $\SO(V)$ for a cleaner statement.

\begin{proposition} 
Assume that $K$ has characteristic $0$. The $\bS_\lambda(E)$-isotypic component of the Koszul complex ${\bf F}_\bullet$ gives an $\bO(V)$-equivariant resolution $0 \to C(\lambda)_\bullet \to \bS_{[\lambda]} V^* \to 0$ where
\[
C(\lambda)_i = \bigoplus_{\mu \vdash 2i,\ \mu \in Q_{1}} \bS_{\lambda / \mu} V^*.
\]
\end{proposition}

\begin{remark} 
When $m > n$, the homology of the Koszul complex on the invariants $x_{i,j}$ is calculated in \cite{ssw}  (this was first done in \cite{enright} but relies on more involved machinery).
\end{remark}


\section{Littlewood spinor modules.} \label{sec:lwoodspin}

In this section we work over a field $K$ of characteristic $0$.

The content of this section is related to Koike's work on developing an analogue of Weyl's construction of spinor modules \cite{koikespin}. While the construction and method seems to be new, the resulting combinatorial formulas were known to Littlewood (see Remarks~\ref{rmk:inv-evenspin} and \ref{rmk:inv-oddspin}).

\subsection{$\rD_n$ spinor modules.}

Let $V$ be a $2n$-dimensional orthogonal space. Let $X_{\pm} = \OGr^{\pm}(n, V)$ be a spinor variety from \S\ref{sec:typeDn}, i.e., a connected component of the space of $n$-dimensional isotropic subspaces of $V$. Let $\cL_\pm$ be the ample generator of $\Pic(X_\pm)$ so that $\cL_\pm^2 = \det \cR_\pm^*$, where $\cR_\pm$ is the rank $n$ tautological subbundle on $X_\pm$. Set $\delta_+ = (1^n)/2$ and $\delta_- = (1^{n-1},-1)/2$. Let $\Delta_{\pm} = V_{\delta_{\pm}}$ be the two half-spinor modules for $\SO(V)$ and set $\Delta = \Delta_+ \oplus \Delta_-$. For our coordinate system on weights, see \S\ref{sec:BWB}.

Let $E$ be an $n$-dimensional vector space. We define the Littlewood half-spinor modules $M_\pm$ to be the pushforwards of $\Sym(E \otimes \cR_\pm^*) \otimes \cL_{\pm}$ to $\hom(E,V)$. It carries an action of $\Spin(V)$ that does not descend to an action of $\SO(V)$ and we have
\[
M_\pm = \rH^0(\OGr^{\pm}(n,V); \bigoplus_\lambda \bS_\lambda E \otimes \bS_\lambda \cR_\pm^* \otimes \cL_\pm) = \bigoplus_\lambda \bS_\lambda E \otimes V_{\lambda + \delta_{\pm}}.
\]
The first equality is \eqref{eqn:cauchy-filt}, and the second equality follows from \S\ref{sec:schurmodules} and an adaptation of the argument used for \eqref{eqn:typeC-sheaf-coh}.

Set $A = \Sym(E \otimes V)$.

\begin{lemma} \label{lemma:dnspincoh}
Let $\lambda$ be a partition with $\ell(\lambda) \le n$ and $\lambda_1 \le n$. Then $\bS_\lambda \cR_\pm \otimes \cL_\pm$ has nonzero cohomology if and only if $\lambda = \lambda^\dagger$. The non-vanishing cohomology is
\begin{align*}
\rH^{(|\lambda|- \rank(\lambda))/2}(X_+; \bS_\lambda \cR_+ \otimes \cL_+) &= 
\begin{cases}
\Delta_+ & \text{if $\rank(\lambda)$ is even}\\
\Delta_- & \text{if $\rank(\lambda)$ is odd}
\end{cases}\\
\rH^{(|\lambda|- \rank(\lambda))/2}(X_-; \bS_\lambda \cR_- \otimes \cL_-) &= 
\begin{cases}
\Delta_- & \text{if $\rank(\lambda)$ is even}\\
\Delta_+ & \text{if $\rank(\lambda)$ is odd}.
\end{cases}
\end{align*}
\end{lemma}

\begin{proof}
We will focus on the statement for $X_+$, as the statement for $X_-$ is similar.

We have $\rho = (n-1, n-2, \dots, 1, 0)$ as defined in \S\ref{sec:BWB}. We need to check if the sequence
\[
(-\lambda_n, \dots, -\lambda_1) + \rho + \delta_+ = (-\lambda_n + \frac{1}{2} + n-1, \dots, -\lambda_1 + \frac{1}{2})
\]
lies on a reflection hyperplane of the Weyl group of type $\rD_n$, i.e., if any two entries are equal up to sign (see \cite[\S 12.1]{humphreys}). Suppose that it does lie on a reflection hyperplane. Since these entries are strictly decreasing, we necessarily have two entries that are negatives of one another, i.e., we have 
\begin{align} \label{eqn:Dnviolation}
\lambda_i + \lambda_j + 1 = i + j
\end{align}
for some $i<j$. We wish to show that this implies that $\lambda \ne \lambda^\dagger$. Without loss of generality, we may assume $i=1$ by removing the hooks for the first $i-1$ diagonal boxes (this does not affect \eqref{eqn:Dnviolation} nor the property of being self-transpose). Then we have a partition such that $\lambda_1 + \lambda_j = j$ for some $j$. If $\lambda_j = 0$, then $\lambda_1 = j$ but $\lambda^\dagger_1 < j$, so $\lambda$ is not self-transpose. Otherwise, if $\lambda_j > 0$, then we must have $\lambda_1, \lambda_j \le 1$, so in particular, $\lambda = (1,1)$ and again is not self-transpose.

Conversely, we have to show that if $\lambda$ is not self-transpose, then \eqref{eqn:Dnviolation} holds for some $1 \le i<j \le n$. Let $i$ be minimal such that $\lambda_i \ne \lambda^\dagger_i$. Again, by removing hooks for the first $i-1$ diagonals, we may assume that $i=1$. If $\lambda_1 > \lambda^\dagger_1$, then we take $j=\lambda_1$ and \eqref{eqn:Dnviolation} holds. Otherwise, we have $\lambda_1 < \lambda^\dagger_1$. If $\lambda_{\lambda_1 + 1} = 1$, then we can take $j = \lambda_1 + 1$. Otherwise, we have $\lambda_{\lambda_1 + 1} > 1$, which implies that $\lambda^\dagger_2 > \lambda_1$. Since $\lambda_2 \le \lambda_1$, if we remove the hook on the first diagonal, we are left with another partition $\mu$ which is not self-dual, and we can replace $\lambda$ by $\mu$ from the point of view of showing that \eqref{eqn:Dnviolation} holds. Hence we are done by induction on the rank of $\lambda$.

Now suppose that $\lambda = \lambda^\dagger$. We need to apply Bott's algorithm. The shifted action of the simple reflections are as follows: for $1 \le i < n$, we have $s_i \colon (\dots, x_i, x_{i+1}, \dots) \mapsto (\dots, x_{i+1}-1, x_i + 1, \dots)$ and $s_n \colon (\dots, x_{n-1}, x_n) \mapsto (\dots, -x_n -1, -x_{n-1} - 1)$. A sequence is dominant in type D if and only if $x_1 \ge \cdots \ge x_{n-1} \ge |x_n|$.

Now we apply Bott's algorithm to $(-\lambda_n + \frac{1}{2}, \dots, -\lambda_1 + \frac{1}{2})$. This weight is dominant if and only if $1 \ge \lambda_1 + \lambda_2$, in which case the lemma is clear. Otherwise, we apply $s_n$ to get $(-\lambda_n + \frac{1}{2}, \dots, \lambda_1 - \frac{3}{2}, \lambda_2 - \frac{3}{2})$. Since $\lambda_1 + \lambda_2 \ge 3$, we get $\lambda_1 - \frac{3}{2} \ge |\lambda_2 - \frac{3}{2}|$. Now we apply $s_{n-2}, s_{n-3}, \dots, s_{n-\lambda_1+1}$ in order. If $\lambda_2 = 1$, then in particular we have $\lambda_3 = \cdots = \lambda_{\lambda_1} = 1$, so the result is $(\frac{1}{2}, \dots, \frac{1}{2}, -\frac{1}{2})$, and we are done. In this case, we used $\lambda_1 - 1 = (|\lambda| - \rank(\lambda))/2$ reflections, so the lemma holds.

Otherwise, we have $\lambda_2 > 1$. Then we apply $s_{n-1}, \dots, s_{n-\lambda_2+2}$ in order. The end result is $-\mu^{\rm op} + \delta_+$ where $\mu$ is the partition obtained from $\lambda$ by removing the hooks at the first two diagonals. In total, we used $\lambda_1 + \lambda_2 - 3$ reflections. We are done by induction since
\[
\frac{|\lambda| - \rank(\lambda)}{2} = \lambda_1 + \lambda_2 - 3 + \frac{|\mu| - \rank(\mu)}{2}. \qedhere
\]
\end{proof}

\begin{proposition}
The minimal free resolution $\bF_\bullet$ of $M_+$ is given by
\[
\bF_i = \bigoplus_{\substack{\lambda = \lambda^\dagger\\ i = (|\lambda| + \rank(\lambda))/2}} \bS_\lambda E \otimes \Delta_{\rank(\lambda)} \otimes A(-|\lambda|)
\]
where $\Delta_d = \Delta_+$ if $d$ is even and $\Delta_d = \Delta_-$ if $d$ is odd. An analogous statement for $M_-$ holds where the definition of $\Delta_d$ is opposite.
\end{proposition}

\begin{proof}
In the notation of \S\ref{sec:geomtech}, $\xi = E \otimes \cR^\perp_+$ (orthogonal complement with respect to the quadratic form on $V \times X_+$) and so 
\[
\bF_i = \bigoplus_{j \ge 0} \rH^j(\OGr^+(n,V); \bigwedge^{i+j}(E \otimes \cR^\perp_+)) \otimes A(-i-j)
\]
by Theorem~\ref{thm:geom-tech}. Now use that $\bigwedge^d(E \otimes \cR^\perp_+) = \bigoplus_{|\lambda|=d} \bS_{\lambda^\dagger} E \otimes \bS_\lambda (\cR^\perp_+)$ \cite[Theorem 2.3.2]{weyman}. We have $\cR^\perp_+ = \cR_+$, so we can use Lemma~\ref{lemma:dnspincoh} to calculate the cohomology above.
\end{proof}

We define the {\bf Littlewood spinor module} by $M = M_+ \oplus M_-$.

\begin{corollary} \label{cor:dnspinres}
The minimal free resolution $\bF_\bullet$ of $M$ is given by
\[
\bF_i = \bigoplus_{\substack{\lambda = \lambda^\dagger\\ i = (|\lambda| + \rank(\lambda))/2}} \bS_\lambda E \otimes \Delta \otimes A(-|\lambda|).
\]
\end{corollary}

\begin{remark} \label{rmk:inv-evenspin}
Taking the $\bS_\lambda(E)$-isotypic component of the minimal free resolution of $M_+$ with respect to $\GL(E)$, we get a complex $C(\lambda)_\bullet$ with 
\[
C(\lambda)_i = \bigoplus_{\substack{\mu = \mu^\dagger\\ i = (|\mu| + \rank(\mu))/2}} \bS_{\lambda/\mu}(V) \otimes \Delta_{\rank(\mu)}
\]
and an exact sequence $0 \to C(\lambda)_\bullet \to V_{\lambda + \delta} \to 0$. The Euler characteristic of this exact sequence gives an inversion formula for the matrix which encodes the decomposition of modules $\Delta_{\pm} \otimes \bS_\lambda V$ into irreducible representations of $\Spin(V)$. A similar formula can be obtained by using $M_-$ instead of $M_+$. These formulas already appear in \cite[11.11.VII]{littlewoodbook}.
\end{remark}

\begin{example} When $n=2$, the resolution is (we just write $\lambda$ in place of $\bS_\lambda E \otimes \Delta \otimes A(-|\lambda|)$):
\[
0 \to (2,2) \to (2,1) \to (1) \to \emptyset \to M \to 0. 
\]
When $n=3$, the resolution is 
\[
0 \to (3,3,3) \to (3,3,2) \to (3,2,1) \to (2,2) \oplus (3,1,1) \to (2,1) \to (1) \to \emptyset \to M \to 0. \qedhere
\]
\end{example}

\subsection{$\rB_n$ spinor modules.}

Now let $V$ be a $2n+1$ dimensional orthogonal space. Let $E$ be a vector space of dimension $n$. Let $\cR$ be the rank $n$ tautological subbundle on the orthogonal Grassmannian $\OGr(n,V)$. Then $\det \cR$ is the square of a line bundle; call its dual $\cL$. So we have $\cL^2 = \det \cR^*$. The {\bf Littlewood spinor module} $M$ is the pushforward of $\Sym(E \otimes \cR^*) \otimes \cL$. It carries an action of $\Spin(V)$ that does not descend to an action of $\SO(V)$. We have
\[
M = \rH^0(\OGr(n,V); \bigoplus_\lambda \bS_\lambda E \otimes \bS_\lambda \cR^* \otimes \cL) = \bigoplus_\lambda \bS_\lambda E \otimes V_{\lambda + \delta}
\]
where $\delta = (1^n)/2$. The first equality is \eqref{eqn:cauchy-filt}, and the second equality follows from \S\ref{sec:schurmodules} and an adaptation of the argument used for \eqref{eqn:typeC-sheaf-coh}.

Let $\Delta = V_\delta$ be the fundamental spin representation. 

\begin{lemma} \label{lem:Bnspin-coh}
Let $\lambda$ be a partition with at most $n$ parts. Then $\bS_\lambda \cR \otimes \cL$ has nonzero cohomology if and only if $\lambda \in Q_{1}$ (see \S\ref{sec:typeBn}). In this case, we have $\rH^{|\lambda|/2}(\OGr(n,V); \bS_\lambda \cR \otimes \cL) = \Delta$.
\end{lemma}

\begin{proof}
Here we have $\rho = (2n-1, 2n-3, \dots, 3, 1)/2$. We need to check if the sequence $(-\lambda_n, \dots, -\lambda_1) + \rho + \delta = (-\lambda_n + n, -\lambda_{n-1} + n-1, \dots, -\lambda_1 + 1)$ lies on a type $\rB_n$ Weyl group hyperplane, i.e., if any entry is 0 or if two entries are equal up to a sign \cite[\S 12.1]{humphreys}.

Since the Weyl groups in type $\rB_n$ and $\rC_n$ are the same and $\rho = \rho' - \delta$ (where $\rho'$ is half of the sum of the positive roots in type $\rC_n$), we are in the exact same setup as calculating cohomology for the Littlewood variety of the symplectic group. So the combinatorics is the same, except for the shift by $\delta$ (which is why we get $\Delta$ instead of a trivial representation). To elaborate: if we were to calculate the minimal free resolution $\bF_\bullet$ from \S\ref{sec:typeCn} using the methods of \S\ref{sec:geomtech}, then we would take $\xi = E \otimes \cR$ over $\IGr(n,V)$. Then $\bigwedge^i \xi = \bigoplus_{|\lambda|=i} \bS_\lambda E \otimes \bS_{\lambda^\dagger} \cR$ \cite[Theorem 2.3.2(b)]{weyman}, and we know that the only $\bS_\lambda E$ that appear in $\bF_\bullet$ are $\lambda \in Q_{-1}$ (and hence the cohomology of $\bS_\mu \cR$ is nonzero exactly for $\mu \in Q_1$ by definition).
\end{proof}

\begin{lemma} \label{lem:spin-can}
The canonical bundle on $X$ is $\omega_X = (\det \cR)^n$.
\end{lemma}

\begin{proof}
We embed $X$ into $Z = \Gr(n,V)$. Let $\cR$ denote also the tautological subbundle on $Z$. By \cite[Proposition 4.3.8]{weyman}, we have a locally free resolution
\[
0 \to \bigwedge^{\binom{n+1}{2}}(\Sym^2 \cR) \to \cdots \to \bigwedge^2(\Sym^2 \cR) \to \Sym^2 \cR \to \cO_Z \to \cO_X \to 0.
\]
But $\bigwedge^{\binom{n+1}{2}}(\Sym^2 \cR) = (\det \cR)^{n+1}$ and $\omega_Z = (\det \cR)^{2n+1}$, so we conclude that $\omega_X$ is the restriction of $(\det \cR)^n$ \cite[Theorem 21.15]{eisenbud}.
\end{proof}

Using the notation from \S\ref{sec:geomtech}, the bundle $\xi = E \otimes \cR^\perp$ (orthogonal complement with respect to the quadratic form on $V \times X$) is an extension
\[
0 \to E \otimes \cR \to \xi \to E \to 0.
\]
Let $\xi' = E \otimes (\cR \oplus \cO)$ be its semisimplification. Using \S\ref{sec:geomtech}, we can construct a (not necessarily minimal) free resolution $\tilde{\bF}_\bullet$ for $M$ with terms
\[
\tilde{\bF}_i = \bigoplus_j \rH^j(\bigwedge^{i+j} \xi' \otimes \cL) \otimes A(-i-j) = \bigoplus_{k \ge 0} \bigwedge^k E \otimes (\bigoplus_{\lambda \in Q_{-1}(2i-2k)} \bS_\lambda E \otimes \Delta) \otimes A(-2i)
\]
(for the second equality, we use \cite[Theorem 2.3.2]{weyman} to get 
\[
\bigwedge^d \xi' = \bigoplus_{d \ge k \ge 0} \bigwedge^k E \otimes (\bigoplus_{|\lambda|=d-k} \bS_\lambda E \otimes \bS_{\lambda^\dagger} \cR)
\]
and Lemma~\ref{lem:Bnspin-coh} to calculate the cohomology of $\bS_{\lambda^\dagger} \cR$).
So a (not necessarily minimal) presentation of $M$ is given by
\begin{align} \label{eqn:nonmin-pres}
\begin{array}{c} E \otimes \Delta \otimes A(-1) \\ \bigwedge^2 E \otimes \Delta \otimes A(-2) \end{array} \to \Delta \otimes A \to M \to 0.
\end{align}

\begin{proposition}
The minimal free resolution $\bF_\bullet$ of $M$ is given by
\[
\bF_i = \bigoplus_{\substack{\lambda = \lambda^\dagger \\ i = (|\lambda| + \rank(\lambda))/2}} \bS_\lambda E \otimes \Delta.
\]
\end{proposition}

\begin{proof}
First we claim that $M$ is a Cohen--Macaulay module. The higher direct images of $\Sym(E \otimes \cR^*) \otimes \cL$ vanish, so by Theorem~\ref{thm:CM-geom}, we need to show that the higher direct images of $\omega_X \otimes \det \xi^* \otimes \cL^* \otimes \Sym(E \otimes \cR^*)$ also vanish. We calculate that $\det \xi^* = (\det \cR^*)^n$, and we have $\omega_X = (\det \cR)^n$ by Lemma~\ref{lem:spin-can}. So we are reduced to showing that $\bS_\lambda \cR^* \otimes \cL^*$ never has higher cohomology. If $\lambda_n > 0$, then writing $\lambda = \mu + (1^n)$, we have $\bS_\lambda \cR^* \otimes \cL^* = \bS_\mu \cR^* \otimes \cL$, which we have already seen has no higher cohomology. Otherwise, we have $\lambda_n = 0$ and the relevant sequence is $(\lambda_1, \dots, \lambda_{n-1}, 0) - \delta$. But this lies on a reflection hyperplane, namely for the simple reflection $s_n$. So the bundle has no cohomology in this situation. This proves that $M$ is Cohen--Macaulay.

Let $W \subset V$ be a $2n$-dimensional subspace on which the orthogonal form has full rank. Write $A = \Sym(E \otimes V)$ and $B = \Sym(E \otimes W)$. The Littlewood variety has codimension $\binom{n+1}{2}$ in $\hom(E,V)$ and its intersection with $\hom(E,W)$ still has codimension $\binom{n+1}{2}$. Hence by generic perfection \cite[Theorem 1.2.14]{weyman}, $\tilde{\bF}_\bullet \otimes_A B$ is a free resolution of $M' = M \otimes_A B$. Finally, we use that the restriction of $\Delta$ to $\Spin(W)$ is $\Delta_+ \oplus \Delta_-$, so we continue to use $\Delta$ as notation after restricting. So by \eqref{eqn:nonmin-pres}, $M'$ has a presentation
\[
\begin{array}{c} E \otimes \Delta \otimes B(-1) \\ \bigwedge^2 E \otimes \Delta \otimes B(-2) \end{array} \to \Delta \otimes B \to M' \to 0.
\]
A calculation with weights (using, for example, \cite[\S 20.1]{fultonharris}) shows that $\Delta_{\pm}$ appears in $V \otimes \Delta_{\mp}$ with multiplicity $1$. So the map $E \otimes \Delta \otimes B(-1) \to \Delta \otimes B$ is essentially unique up to a choice of scalars, and from Corollary~\ref{cor:dnspinres}, it does not have any linear syzygies. Hence the quadratic relations are redundant (since the linear syzygies existing in $\tilde{\bF}_2$ cancel with these quadratic relations when we minimize the complex) and we conclude that $M'$ is the Littlewood spinor module for $B$. In particular, the Betti tables of $M$ and the Littlewood spinor module for $B$ are the same, which gives the conclusion.
\end{proof}

\begin{remark} \label{rmk:inv-oddspin}
Taking isotypic components of the minimal free resolution of $M$ with respect to $\GL(E)$, we get complexes $C(\lambda)_\bullet$ with 
\[
C(\lambda)_i = \bigoplus_{\substack{\mu = \mu^\dagger\\ i = (|\mu| + \rank(\mu))/2}} \bS_{\lambda/\mu}(V) \otimes \Delta
\]
and an exact sequence $0 \to C(\lambda)_\bullet \to V_{\lambda + \delta} \to 0$. The Euler characteristic of this exact sequence gives an inversion formula for the matrix which encodes the decomposition of modules $\Delta \otimes \bS_\lambda V$ into irreducible representations of $\Spin(V)$. These formulas already appear in \cite[11.11.VII]{littlewoodbook}.
\end{remark}

\begin{corollary}
Let $\lambda$ be a partition with $\ell(\lambda) \le n+1$ and $\lambda_1 \le n$. Then $\bS_\lambda \cR^\perp \otimes \cL$ has nonzero cohomology if and only if $\lambda = \lambda^\dagger$. In this case, we have 
\begin{align*}
\rH^i(X; \bS_\lambda \cR^\perp \otimes \cL) = \begin{cases} \Delta & \text{if } i=\frac{|\lambda|-\rank(\lambda)}{2}\\
0 & \text{else} \end{cases}.
\end{align*}
\end{corollary}

\begin{proof}
From \S\ref{sec:geomtech}, if $\bF_\bullet$ is the minimal free resolution of $M$, then we have
\[
\bF_i = \bigoplus_{j \ge 0} \rH^j(X; \bigwedge^{i+j}(E \otimes \cR^\perp) \otimes \cL) \otimes A(-i-j),
\]
and $\bigwedge^{i+j}(E \otimes \cR^\perp) = \bigoplus_{|\lambda|=i+j} \bS_{\lambda^\dagger} E \otimes \bS_\lambda \cR^\perp$ \cite[Theorem 2.3.2]{weyman}, but we have calculated $\bF_i$ in another way. Then the result follows by comparing the multiplicity spaces of the $\GL(E)$ action.
\end{proof}

\begin{remark}
A similar technique can be used to calculate the cohomology of $\bS_\lambda \cR^\perp$.
\end{remark}


\section{Type ${\rm G}_2$.}

\subsection{Homogeneous spaces.}

We start with the description of the homogeneous spaces for the group of type ${\rm G}_2$. We work over a field $K$ of arbitrary characteristic. We follow \cite[\S 6]{anderson} for definitions and basic properties. Let $V$ be a 7-dimensional vector space and let $x_1, \dots, x_7$ be a basis for $V^*$. Define $\gamma_0 \in \bigwedge^3 V^*$ and $q_0 \in \Sym^2(V^*)$ by
\begin{align*}
  \gamma_0 &= x_1\wedge x_4\wedge x_7 + x_2\wedge x_4\wedge x_6 +
  x_3\wedge x_4\wedge x_5 - x_2\wedge x_3\wedge x_7 - x_1\wedge
  x_5\wedge x_6\\
  q_0 &= x_1x_7 + x_2x_6 + x_3x_5 + x_4^2.
\end{align*}
Given another pair $(\gamma, q) \in \bigwedge^3 V^* \oplus
\Sym^2(V^*)$, we say that $(\gamma, q)$ is a compatible nondegenerate
pair if it is equivalent to $(\gamma_0, q_0)$ after some change of
basis \cite[Definition 6.1.2]{anderson}. We define a bilinear
orthogonal from $\beta$ via 
\[
\beta(x,y) = q(x+y) - q(x) - q(y).
\]

Let $G$ be the automorphism group of a compatible nondegenerate pair $(\gamma, q)$. Then $G$ is a simple algebraic group of type ${\rm G}_2$ \cite[Proposition 6.1.5]{anderson}. We label the Dynkin diagram of $G$ using Bourbaki notation
\[
1 \Lleftarrow 2.
\]

The simple roots of $G$ are $\alpha_1$ and $\alpha_2$, and we set the
fundamental weights $\omega_1 = 2\alpha_1 + \alpha_2$, $\omega_2 =
3\alpha_1 + 2\alpha_2$. In particular, $V$ has highest weight
$\omega_1$, and the adjoint representation has highest weight
$\omega_2$. Let $V_\lambda$ be the Schur module with highest weight
$\lambda_1\omega_1 + \lambda_2\omega_2$. Note that $V_{(\lambda_1 -
  \lambda_2, \lambda_2)}$ is a quotient of $\bS_\lambda V$ with
multiplicity 1, and is obtained via an analogue of Weyl's
construction, see \cite{huang}. We set 
\[
\bS_{[\lambda]} V = V_{(\lambda_1 - \lambda_2, \lambda_2)}.
\]

A 2-dimensional subspace $W \subset V$ is {\bf isotropic} if
$\gamma(w,w',v) = 0$ for all $w,w' \in W$ and $v \in V$. A
1-dimensional subspace is {\bf isotropic} if it belongs to some
2-dimensional isotropic subspace. Then $G$ acts transitively on the 1-
and 2-dimensional isotropic subspaces of $V$, so we get two
Grassmannians $\Gr_\gamma(1,V)$ and $\Gr_\gamma(2,V)$, and a flag
variety $\Fl_\gamma(V) = \Fl_\gamma(1,2;V)$, which have dimensions 5,
5, and 6, respectively.

Then we get tautological subbundles $\cR_1 \subset \cR_2 \subset V$
over $\Fl_\gamma(V)$. Given an isotropic vector $u \in V$, the
subspace $V_u = \{ v \mid \langle u,v \rangle \in \Gr_\gamma(2,V) \}$
is 3-dimensional and $q$-isotropic \cite[Lemma 6.1.8]{anderson}. Hence
we get a subbundle $\cR_3 \subset V$ containing $\cR_2$. Then one
defines $\cR_{7-i}$ to be the $\beta$-orthogonal complement of $\cR_i$
for $i=1,2,3$. In this way, we get a complete flag
\[
\cR_1 \subset \cR_2 \subset \cR_3 \subset \cR_4 \subset \cR_5 \subset
\cR_6 \subset \cR_7 = V
\]
of vector bundles over $\Fl_\gamma(V)$. We let $\cL(a,b)$ be the line bundle whose sections are $V_{(a,b)}^* \cong V_{(a,b)}$ when $a,b \ge 0$. In particular, the higher cohomology of $\cL(a,b)$ vanishes by Kempf vanishing (see \S\ref{sec:schurmodules}).

\begin{proposition} \label{prop:G2linebundles}
With the identifications as above we have
\begin{align*}
  \cR_1 &\cong \cL(-1,0), \quad &\cR_2/\cR_1 &\cong \cL(1,-1),
  &\cR_3/\cR_2 &\cong \cL(-2,1), \quad &\cR_4/\cR_3 &\cong \cL(0,0),\\
  \cR_5/\cR_4 &\cong \cL(2,-1), \quad &\cR_6/\cR_5 &\cong \cL(-1,1),
  &\cR_7/\cR_6 &\cong \cL(1,0).
\end{align*}
\end{proposition}

\begin{proof} 
The first isomorphism comes from the fact that the space of $\gamma$-isotropic lines embeds in $\bP(V_{\omega_1})$, and the second from the fact that the space of $\gamma$-isotropic planes embeds in $\bP(V_{\omega_2})$ via the line bundle $(\bigwedge^2 \cR_2)^{-1}$. The third comes from the fact that $\cR_2 / \cR_1 \otimes \cR_3 / \cR_2 \cong \cR_1$ \cite[Lemma 6.2.3]{anderson}. For the rest, we use the duality induced by $\beta$.
\end{proof}

Using this, we can calculate cohomology of Schur functors $\bS_\lambda (\cR_r^*)$ on $\Gr_\gamma(r,V)$. The explicit descriptions of the spaces $\Gr_\gamma(i,V)$ gives that the projection $\pi_1 \colon \Fl_\gamma(V) \to \Gr_\gamma(1,V)$ is the projective bundle $\bP(\cR_3 / \cR_1)$, and the projection $\pi_2 \colon \Fl_\gamma(V) \to \Gr_\gamma(2,V)$ is the projective bundle $\bP(\cR_2)$.

\begin{proposition} \label{prop:G2cohomology}
For $\lambda_1 \ge \lambda_2 \ge 0$, we have 
\begin{align*}
\rR^i {\pi_2}_* \cL(\lambda_1 - \lambda_2, \lambda_2) &= \begin{cases} \bS_\lambda(\cR_2^*) & i=0 \\
0 & i>0 \end{cases}.
\end{align*} 
In particular, the cohomology of $\cL(\lambda_1 - \lambda_2, \lambda_2)$ and
$\bS_\lambda(\cR_2^*)$ agree. So we get 
\[
\rH^i(\Gr_\gamma(2,V); \bS_\lambda(\cR_2^*)) = \begin{cases} \bS_{[\lambda]} V & i=0 \\ 0 & i>0 \end{cases}.
\]
\end{proposition}

\begin{proof}
Let $\tilde{\cR}_1$ and $\tilde{\cR}_2$ denote the tautological bundles on $\Fl_\gamma(V)$ to distinguish from the tautological bundle $\cR_2$ on $\Gr_\gamma(2,V)$. From Proposition~\ref{prop:G2linebundles}, we have $\tilde{\cR}_1^* = \cL(1,0)$ and $\det(\tilde{\cR}_2^*) = \cL(0,1)$. In particular, 
\[
\cL(\lambda_1 - \lambda_2, \lambda_2) = (\tilde{\cR}_1^*)^{\otimes (\lambda_1 - \lambda_2)} \otimes (\det\tilde{\cR}^*_2)^{\otimes \lambda_2}
\]
Note that $\det \tilde{\cR}_2^* = \pi_2^*(\det \cR_2^*)$, so by the projection formula, we get 
\[
\rR^i {\pi_2}_* \cL(\lambda_1 - \lambda_2, \lambda_2) = \rR^i {\pi_2}_* (\tilde{\cR}_1^*)^{\otimes (\lambda_1 - \lambda_2)} \otimes (\det \cR_2^*)^{\otimes \lambda_2}
\]
Finally, since $\pi_2$ is a relative $\bP^1$ and $\tilde{\cR}_1$ serves as $\cO(-1)$ for this projective bundle, we get that $\cR^i {\pi_2}_* (\tilde{\cR}^*_1)^{\otimes (\lambda_1 - \lambda_2)}$ is $\Sym^{\lambda_1 - \lambda_2}(\cR_2^*)$ if $i=0$ and is $0$ otherwise. Now use the fact that $\rank \cR_2^* = 2$ so that $\Sym^{\lambda_1 - \lambda_2}(\cR_2^*) \otimes (\det \cR_2^*)^{\otimes \lambda_2} = \bS_\lambda(\cR_2^*)$.

The fact that the cohomology of $\cL(\lambda_1 - \lambda_2, \lambda_2)$ and $\bS_\lambda(\cR_2^*)$ agree now follows from the Leray spectral sequence.
\end{proof}

Using a similar (but easier) argument, we can show that 
\[
\rH^i(\Gr_\gamma(1,V); (\cR_1^*)^{\otimes d}) = \begin{cases} \bS_{[d]} V & i=0 \\ 0 & i>0 \end{cases}.
\]

\subsection{Analogue of determinantal varieties}

Let $E$ be a 2-dimensional vector space and let $X = \hom(E,V)$. For
$0 \le r \le 2$, let $Y_r \subset X$ be the subvariety of linear maps
$\phi$ such that $\phi(E)$ is a $\gamma$-isotropic subspace and $\rank
\phi \le r$. There are 3 orbits of $\GL(E) \times G$ acting on $Y_2$,
and they are classified by the rank of $\phi$.

The Lie algebra $\mathfrak{so}(10)$ contains a subalgebra isomorphic to $\mathfrak{sl}(2) \times \mathfrak{so}(8)$, and $\mathfrak{so}(8)$ contains $\mathfrak{g}_2$ via triality (see \cite[\S 5]{adams} or \cite[\S 20.3]{fultonharris}). Let ${\bf spin}^+(10)$ be one of the half-spinor representations of $\mathfrak{so}(10)$. The highest weight orbit is the spinor variety, which we denoted $\OGr^+(5,10)$ in \S\ref{sec:typeDn}.

The restriction of ${\bf spin}^+(10)$ to $\mathfrak{sl}(2) \times \mathfrak{g}_2$ is $E \otimes (V + K)$. Since we know a description of
the orbits in $E \otimes V$, we get
\begin{align} \label{eqn:spinorsection}
\bP(Y_2) = (E \otimes V) \cap \OGr^+(5,10),
\end{align}
where $\bP(Y_2)$ denotes the projectivization of $Y_2$. 

\begin{theorem}
\begin{compactenum}[\rm (1)]
\item The variety $Y_2$ is irreducible and has codimension $5$. Furthermore, $Y_2$ is Gorenstein and has rational singularities. The variety $Y_1$ is irreducible and has codimension $7$. Furthermore, $Y_1$ has rational singularities.

\item The coordinate ring of $Y_r$ has a filtration with associated graded given by
\[
K[Y] \approx \bigoplus_{\lambda,\ \ell(\lambda) \le r} \bS_\lambda E \otimes \bS_{[\lambda]} V^*.
\]

\item The ideal defining $Y_2$ is minimally generated by $10$ quadrics and the ideal defining $Y_1$ is minimally generated by $24$ quadrics (more details in the proof).

\item The singular locus of both $Y_2$ and $Y_1$ is $Y_0 = \{0\}$.
\end{compactenum}
\end{theorem}

\begin{proof}
The case $r=0$ is trivial, so we focus on $r=1,2$. The variety $Y_r$ has a desingularization by the bundle $\hom(E,\cR_r)$ over $\Gr_\gamma(r,V)$. This shows that $Y_r$ is irreducible with the claimed codimension. Let $\tilde{Y}_r$ denote the normalization of $Y_r$. By \S\ref{sec:geomtech}, we have
\[
K[\tilde{Y}_r] = \rH^0(\Gr_\gamma(r,V); \Sym(E \otimes \cR^*_r))
\]
The multiplication map
\[
\Sym^i(E \otimes \cR^*_r) \otimes (E \otimes \cR^*_r) \to \Sym^{i+1}(E \otimes \cR^*_r)
\]
is surjective and its kernel has a good filtration. In particular the kernel does not have higher cohomology by Proposition~\ref{prop:G2cohomology}, so taking sections of the multiplication map preserves surjectivity. In particular, taking sections gives the multiplication map for $K[\tilde{Y}_r]$, and we conclude that it is generated over $A$ in degree 0. Hence $\tilde{Y}_r = Y_r$, and $Y_r$ is normal. Finally, $\rH^i(\Gr_\gamma(r,V); \Sym(E \otimes \cR_r^*)) = 0$ for $i>0$ (use \eqref{eqn:cauchy-filt} to get $\Sym(E \otimes \cR_r^*) \approx \bigoplus_{\ell(\lambda) \le r} \bS_\lambda(E) \otimes \bS_\lambda(\cR_r^*)$ and then apply Proposition~\ref{prop:G2cohomology}), so $Y_r$ has rational singularities.

Via a Hilbert series calculation, one can show that \eqref{eqn:spinorsection} is a reduced complete intersection. In particular, $Y_2$ is Gorenstein, and the ideal of $Y_2$ is generated by quadrics. Then it is easy to see that the additional equations needed to cut out $Y_1$ are the $2 \times 2$ minors. Indeed, $\bP(Y_1)$ is the highest weight orbit in $\bP(X)$, which in this case is the Segre embedding of a quadric with $\bP^1$.

Finally, we calculate the singular locus. Since the $\GL(E) \times G$-orbits on $Y_2$ are classified by the rank of the map, the statement about $Y_1$ is clear. We just need to show that $Y_2$ is not singular along $Y_1$. This is done by picking a single rank $1$ map in $Y_1$ and using the Jacobian criterion. We finish by describing the ideal in a basis-free way. 

For $r=1$, the ideal is generated by the $2 \times 2$ minors $\bigwedge^2 E \otimes \bigwedge^2 V^*$ and the polarization of the quadratic form $\Sym^2 E \otimes \langle q \rangle$, both in degree $2$. For $r=2$, the ideal is generated by $\bigwedge^2 E \otimes V^* \subset \bigwedge^2 E \otimes \bigwedge^2 V^*$ and $\Sym^2 E \otimes \langle \beta \rangle$, both in degree $2$. The first set of generators are obtained by using the trilinear form $\gamma$.
\end{proof}

\begin{theorem} \label{thm:G2minfreeres} 
Over a field of characteristic different from $2$, the graded Betti table of $Y_2$ is 
\small \begin{verbatim}
       0  1  2  3  4 5
total: 1 10 16 16 10 1
    0: 1  .  .  .  . .
    1: . 10 16  .  . .
    2: .  .  . 16 10 .
    3: .  .  .  .  . 1
\end{verbatim} \normalsize
Over a field of characteristic $2$, the graded Betti table of $Y_2$ is
\small \begin{verbatim}
       0  1  2  3  4 5
total: 1 10 17 17 10 1
    0: 1  .  .  .  . .
    1: . 10 16  1  . .
    2: .  .  1 16 10 .
    3: .  .  .  .  . 1
\end{verbatim} \normalsize
When the field has characteristic $0$, the terms of the minimal free resolution ${\bf F}_\bullet$ of $Y_2$ are
\begin{align*}
  \begin{array}{ll}
    {\bf F}_0 = A & 
    {\bf F}_1 = (\det E \otimes V \oplus \Sym^2 E) \otimes A(-2)\\
    {\bf F}_2= \det E \otimes E \otimes (K \oplus V) \otimes A(-3) & 
    {\bf F}_3 = (\det E)^2 \otimes E \otimes (K \oplus V) \otimes A(-5)\\
    {\bf F}_4 = (((\det E)^3 \otimes V) \oplus (\det E)^2 \otimes
    \Sym^2 E) \otimes A(-6) & 
    {\bf F}_5 = (\det E)^4 \otimes A(-8).
\end{array}
\end{align*}
\end{theorem}

\begin{proof}
The statement about the graded Betti table follows from \eqref{eqn:spinorsection} and a Macaulay2 calculation. To get the $\GL(E) \times G$-action on the terms in characteristic $0$, we can use the fact that we know the ranks in the graded Betti table and the terms in the coordinate ring of $K[Y_2]$: so if we calculate the Euler characteristic of the complex in a fixed degree (and work by induction on degree), then the representation-theoretic structure of all but one of the terms will be known.
\end{proof}

\begin{corollary}
Over a field of characteristic $0$, the isotypic component of ${\bf F}_\bullet$ is an exact complex $C(\mu)_\bullet$ of representations of ${\rm G}_2$ resolving the representation $V_{[\mu ]}$ by Schur functors on $V$:
\begin{align*}
0 \to \bS_{\mu / (4,4)} V \to (\bS_{\mu / (3,3)} V \otimes V) \oplus \bS_{\mu / (4,2)} V \to \bS_{\mu / (3,2)} V \otimes (K \oplus V) \to \\
\bS_{\mu / (2,1)} V \otimes (K \oplus V) \to (\bS_{\mu / (1,1)} V \otimes V) \oplus \bS_{\mu / (2)} V \to \bS_\mu V \to V_{(\mu_1 - \mu_2, \mu_2)} \to 0.
\end{align*}
\end{corollary}

\begin{proof}
Since $A = \bigoplus_\mu \bS_\mu E \otimes \bS_\mu V$, and the coordinate ring of $Y_2$ is $\bigoplus_\mu \bS_\mu E \otimes \bS_{[\mu]} V$ where $\bS_{[\mu]} V \cong V_{(\mu_1 - \mu_2, \mu_2)}$, this complex is the $\bS_\mu E$-isotypic component of the resolution $\bF_\bullet$ of $K[Y_2]$ in Theorem~\ref{thm:G2minfreeres}.
\end{proof}

For the next result, we abbreviate $\bS_\lambda E \otimes V_\mu
\otimes A(-i)$ by $(\lambda_1, \lambda_2; \mu_1, \mu_2)(-i)$.

\begin{proposition} Over a field of characteristic $0$, the terms of
  the minimal free resolution ${\bf F}_\bullet$ of $Y_1$ are
  \begin{align*}
    {\bf F}_0 &= (0,0;0,0)\\
    {\bf F}_1 &= (2,0;0,0)(-2) + (1,1;1,0)(-2) + (1,1;0,1)(-2) \\
    {\bf F}_2 &= (2,1;0,0)(-3) + (3,0;1,0)(-3) + (2,1;2,0)(-3) \\
    {\bf F}_3 &= (3,1;0,0)(-4) + (2,2;1,0)(-4) + (3,1;1,0)(-4) +
    (3,1;2,0)(-4) + (2,2;0,1)(-4)\\
    {\bf F}_4 &= (4,1;1,0)(-5) + (4,1;0,1)(-5) + (3,3;0,0)(-6) +
    (3,3;2,0)(-6) \\
    {\bf F}_5 &= (3,3;1,0)(-6) + (4,2;1,0)(-6) + (4,3;1,0)(-7) +
    (4,3;0,1)(-7) \\
    {\bf F}_6 &= (4,3;0,0)(-7) + (5,2;0,0)(-7) + (6,2;1,0)(-8) \\
    {\bf F}_7 &= (6,3;0,0)(-9)
  \end{align*}

In particular, the graded Betti table of $Y_1$ is 
\small 
\begin{verbatim}
       0  1  2   3   4  5  6 7
total: 1 24 84 126 119 77 27 4
    0: 1  .  .   .   .  .  . .
    1: . 24 84 126  84 35  6 .
    2: .  .  .   .  35 42 21 4
\end{verbatim} 
\end{proposition}
\normalsize
\begin{proof} 
This can be done with Macaulay2.
\end{proof}


\section{Type ${\rm F}_4$.}

In this section, we will assume that $K$ is a field of characteristic $0$.

\subsection{Description of homogeneous spaces}

For this section, we follow \cite[Example 9.1]{carr}. For $G$ of type ${\rm F}_4$, let $V$ be its $26$ dimensional representation. This is $V_{\omega_4}$, according to the following labeling:
\[
\xymatrix @-1.2pc { 1 \ar@{-}[r] & 2 \ar@{=>}[r] & 3 \ar@{-}[r] & 4}.
\]
As usual, we will take $A = \Sym(E \otimes V)$ throughout, where $E$ will be a vector space whose dimension is specified in each subsection.

We also use the notation $V_{(a,b,c,d)}$ to denote the module with highest weight $a\omega_1 + b\omega_2 + c\omega_3 + d\omega_4$. All finite-dimensional representations of $G$ are self-dual. There is a commutative $G$-invariant multiplication $\Sym^2 V \to V$, which we will denote by $\#$. A subspace $W \subset V$ is {\bf $\#$-isotropic} if $x\# y=0$ for all $x,y \in W$. Then $G/P_i \subset \bP(V_{\omega_i})$ is the space of $\#$-isotropic $d_i$-dimensional subspaces, where $d_i$ is given by the following table:
\[
\begin{array}{c|c|c}
  i & d_i  & \dim G/P_i \\
  \hline
  1 & 6 & 15\\
  2 & 3 & 20\\
  3 & 2 & 20\\
  4 & 1 & 15
\end{array}
\]
So for $i=1$, we have $6$-dimensional $\#$-isotropic subspaces. The Levi subgroup of the parabolic which stabilizes such a subspace is of type ${\rm C}_3$. So such subspaces inherit a symplectic form and a point of $G/B$ is a choice of such a $6$-dimensional space plus a complete isotropic flag inside of it. The other homogeneous spaces are obtained by forgetting some subspaces.

Let $\Fl(V) = G/B$, which has dimension $24$. From the description above, it has a partial flag of tautological subbundles $\cR_1 \subset \cR_2 \subset \cR_3 \subset \cR_6 \subset V$. Then $\cR_6$ is a symplectic bundle, i.e., we have a map $\bigwedge^2 \cR_6 \to \cM$, where $\cM$ is some line bundle. To figure out which one, we first note that $\det \cR_6 = \cM^{\otimes 3}$, and that $\det \cR_6^*$ gives an embedding of $G/P_1$ into $\bP(\bigwedge^6 V)$. But $\bigwedge^6 V$ only contains two summands whose highest weights are multiples of 3: $V_{3\omega_1}$ and $V_{3\omega_4}$, and the Picard group of $G/P_1$ is generated by $\cL(1,0,0,0)$, so we conclude that $\cM = \cL(-1,0,0,0)$. In particular, we can define $\cR_4 = \cR_2^\perp$ and $\cR_5 = \cR_1^\perp$ (orthogonal complements are defined in $\cR_6$). Also, the projection $G/B \to G/P_1$ identifies $G/B$ with the relative isotropic flag variety ${\bf I}\Fl(\cR_6)$.

Let $\cL(a,b,c,d)$ be the line bundle whose sections are $V^*_{(a,b,c,d)} \cong V_{(a,b,c,d)}$. Using the fact that $G/P_3 \subset \bP(V_{\omega_3}) \subset \bP(\bigwedge^2 V)$ and $G/P_2 \subset \bP(V_{\omega_2}) \subset \bP(\bigwedge^3 V)$, where the embeddings are given by $\det \cR^*_i$, we get

\begin{proposition} \label{prop:F4linebundles}
In the above notation we have
\begin{align*}
\cR_1 \cong \cL(0,0,0,-1), &\quad& \cR_2/\cR_1 \cong \cL(0,0,-1,1), &\quad& \cR_3/\cR_2 \cong \cL(0,-1,1,0),\\
\cR_4/\cR_3 \cong \cL(-1,1,-1,0), &\quad& \cR_5/\cR_4 \cong \cL(-1,0,1,-1), &\quad& \cR_6/\cR_5 \cong \cL(-1,0,0,1).
\end{align*}
\end{proposition}

\begin{proof}
The proof is similar to the proof of Proposition~\ref{prop:G2linebundles}.
\end{proof}

\subsection{6 copies of $V$.}

Let $E$ be a 6-dimensional vector space. Let $X = \hom(E, V)$ and let $Y \subset X$ be the subvariety which is the image of the vector bundle $\hom(E, \cR_6)$ over $G/P_1$. (Here we consider the natural map $\hom(E,\cR_6) \subset X \times G/P_1 \to X$.)

Let $U$ be a 6-dimensional symplectic vector space. Given a representation $\bS_\lambda U$ of $\GL(U)$, let $\gamma_{\lambda, \mu}$ denote the multiplicity of $\bS_{[\mu]} U$ upon branching to
$\Sp(U)$, i.e., $\bS_\lambda U = \bigoplus_\mu (\bS_{[\mu]} U)^{\oplus \gamma_{\lambda, \mu}}$.

\begin{proposition} \label{prop:F4cohomology}
Let $\lambda$ be a partition. Then 
\begin{align*}
\rH^i(G/P_1; \bS_\lambda \cR_6^*) = \begin{cases} \bigoplus_\mu V_{((|\lambda| -  |\mu|)/2, \mu_3, \mu_2 - \mu_3, \mu_1 - \mu_2)}^{\oplus \gamma_{\lambda, \mu}} & i=0\\
0 & i>0 \end{cases}.
\end{align*}
\end{proposition}

\begin{proof}
Let $\pi \colon G/B \to G/P_1$ be the projection. Then $\pi$ is the relative symplectic flag variety ${\bf IFl}(\cR_6)$. We have a tautological isotropic flag $0 \subset \cR_1 \subset \cR_2 \subset \cR_3 \subset \pi^* \cR_6$ on $G/B$. In particular, if $\nu_1 \ge \nu_2 \ge \nu_3 \ge 0$, then using the relative version of the Borel--Weil--Bott theorem for symplectic flag varieties (see \cite[Proof of Corollary 4.3.4]{weyman}), we get
\[
\rR^i\pi_*(\cR_1^{\otimes -\nu_1} \otimes (\cR_2/\cR_1)^{\otimes -\nu_2} \otimes (\cR_3/\cR_2)^{\otimes -\nu_3}) = \begin{cases} \bS_{[\nu]}(\cR_6^*) & i=0\\
0 & i>0 \end{cases}.
\]
By Proposition~\ref{prop:F4linebundles}, the line bundle on the left hand side is $\cL(0, \nu_3, \nu_2 - \nu_3, \nu_1 - \nu_2)$.

Since $\cL(d,0,0,0) = \pi^* (\cM^*)^{\otimes d}$, we can use the projection formula to get 
\[
\rR^i\pi_* \cL(d,\nu_3, \nu_2-\nu_3,\nu_1-\nu_2) = \begin{cases} \bS_{[\nu]}(\cR_6^*) \otimes (\cM^*)^{\otimes d} & i=0\\
0 & i>0 \end{cases}.
\]
In particular, the Leray spectral sequence gives
\[
\rH^i(G/P_1, \bS_{[\nu]}(\cR_6^*) \otimes (\cM^*)^{\otimes d}) = 
\begin{cases} V_{(d, \nu_3, \nu_2 - \nu_3, \nu_1 - \nu_2) }& i=0\\
0 & i>0 \end{cases}.
\]
Finally, to finish the proof, we note that $\bS_\lambda(\cR_6^*) = \bigoplus_\mu (\bS_{[\mu]}(\cR_6^*) \otimes (\cM^*)^{\otimes (|\lambda|-|\mu|)/2})^{\oplus \gamma_{\lambda,\mu}}$.
\end{proof}

\begin{theorem} \label{thm:6copiesF4}
The variety $Y$ is the $\GL(E) \times G$-orbit of $E^* \otimes W$ where $W$ is a $6$-dimensional $\#$-isotropic subspace, so the map $\hom(E, \cR_6) \to Y$ is birational. 
The variety $Y$ is normal with rational singularities.  The coordinate ring of $Y$ is
\[
K[Y] = \rH^0(G/P_1; \Sym(E \otimes \cR_6^*)) = \bigoplus_\lambda \bS_\lambda E \otimes (\bigoplus_\mu V^{\oplus \gamma_{\lambda, \mu}}_{((|\lambda| - |\mu|) / 2, \mu_3, \mu_2 - \mu_3, \mu_1 - \mu_2)}).
\]
\end{theorem}

\begin{proof}
Let $\tilde{Y}$ be the normalization of $Y$. By \S\ref{sec:geomtech}, we have
\[
K[\tilde{Y}] = \bigoplus_\lambda \bS_\lambda E \otimes \rH^0(G/P_1; \bS_\lambda \cR^*_6).
\]
Since the bundle $E\otimes \cR_6^*$ is irreducible, we can apply \cite[Proposition 2.2]{landsbergweyman} to conclude that $K[\tilde{Y}]$ is generated as an $A$-module in degree 1. In particular, this implies that $\tilde{Y} = Y$ and $Y$ is normal. Using Proposition~\ref{prop:F4cohomology} and \eqref{eqn:cauchy-filt}, we get the claimed calculation of $K[Y]$, and that the higher cohomology of $\Sym(E \otimes \cR_6^*)$ vanishes (so $Y$ has rational singularities).
\end{proof}

\begin{remark}
We have $\codim Y = 105$, and $\bigwedge^{20} (V/\cR_6)^* \cong \bigwedge^6 \cR_6 \cong \cL(-3,0,0,0)$. Also, 
\[
\rH^{15}(G/P_1; (\bigwedge^{20} (V/\cR_6)^*)^{\otimes 6}) \cong V_{10\omega_1},
\]
so $(\bigwedge^6 E)^{\otimes 20} \otimes V_{10\omega_1} \otimes A(-120)$ is contained in ${\bf F}_{105}$ which means that $Y$ is not Gorenstein.
\end{remark}

\begin{conjecture} 
The Tor modules $\Tor_i^A(K, K[Y])$ are Schur functors on $E$ and $V$.
\end{conjecture}

\subsection{3 copies of $V$.}

Let us consider the case $\dim E = 3$ at the cost of only focusing on representations with highest weights of the form $b\omega_2 + c\omega_3 + d\omega_4$. So let $X = \hom(E, V)$ and $Y \subset X$ the image of the projection of $\hom(E, \cR_3)$, which is a vector bundle over $G/P_2$. (Here we consider the natural map $\hom(E,\cR_3) \subset X \times G/P_2 \to X$.) If $\ell(\lambda) \le 3$, we will write $\bS_{[\lambda]} V = V_{(0,\lambda_3, \lambda_2 - \lambda_3, \lambda_1 - \lambda_2)}$.

\begin{theorem} 
The variety $Y$ is spherical and normal with rational singularities. The variety $Y$ is the $\GL(E) \times G$-orbit of $E^* \otimes W$ where $W$ is a $3$-dimensional $\#$-isotropic subspace of dimension $3$. So the projection onto $Y$ is birational. The coordinate ring of $Y$ is
\[
K[Y] = \bigoplus_\lambda \bS_\lambda E \otimes \bS_{[\lambda]} V.
\]
\end{theorem}

\begin{proof}
Consider the map $\pi \colon G/B \to G/P_2$ which has fibers $\Fl(\cR_3) \times \bP^1$. A similar argument as in the proof of Proposition~\ref{prop:F4cohomology} shows that $\pi_* \cL(0, \lambda_3, \lambda_2 - \lambda_3, \lambda_1 - \lambda_2) = \bS_\lambda \cR_3^*$ and its higher direct images vanish. Hence
\[
\bigoplus_\lambda \bS_\lambda E \otimes \rH^0(G/P_2; \bS_\lambda \cR_3^*) = \bigoplus_\lambda \bS_\lambda E \otimes V_{(0, \lambda_3,\lambda_2 - \lambda_3, \lambda_1 - \lambda_2)}.
\]
Checking that $Y$ is normal with rational singularities follows as in the proof of Theorem~\ref{thm:6copiesF4}. 
\end{proof}

\begin{theorem}\label{thm:3copf4quad}
The defining ideal $Y$ is generated by quadrics.  The defining equations are given by the representations $\bigwedge^2 E\otimes V_{\omega_1}$ and by $\Sym^2 E\otimes (V_{\omega_4}\oplus K)$.
\end{theorem}

\begin{proof}
Using the notation from \S\ref{sec:geomtech}, we define $\xi$ by the short exact sequence
\[
0 \to \xi \to E \otimes V \to E \otimes \cR_3^* \to 0.
\] 
By Theorem~\ref{thm:geom-tech}, we need to prove that $\rH^i (G/P_2;\bigwedge^{i+1}\xi )=0$ for $i>1$.  The deformation technique of Grosshans (see \S\ref{sec:grosshans}) shows that the only representations that might appear are of the form $\bS_{\mu}E\otimes V_\nu$ where $\mu_1\le 2$: we degenerate to a ring where multiplication is given by Cartan product and hence generated by $\bS_\mu E \otimes V_\nu$ where $\mu_1 \le 1$; also we know from \S\ref{sec:gradedGalg} that the ideal of the degenerate ring is generated by quadrics in the new generators. This last statement implies that the equations in $Y$ must occur in degrees $\le 6$.

We think of homogeneous bundles on $G/P_2$ as rational $P_2$-modules.  Consider the composition series of $\xi$ treated as a $P_2$-module and let $\xi'$ be the corresponding associated graded semi-simple $P_2$-module. Then the cohomology of the exterior powers of $\xi'$ can be calculated by a computer, and 
was performed for us by Witold Kra\'skiewicz. We get lists $L(i)$ of representations that involve Schur functors $\bS_\mu E\otimes V_\nu$ with $\mu_1 \le 2$ that appear in $\rH^i(G/P_2;\bigwedge^{i+1}\xi')$ and $\Sym^i(E \otimes V)$ for $i=5,6$:
\begin{align*}
L(5): & (2,2,1;0,0,0,0), (2,2,1;0,0,0,1), (2,2,1;1,0,0,0), (2,2,1;0,0,1,0)\\
L(6): & (2,2,2;0,0,0,0), (2,2,2;0,0,0,1), (2,2,2;1,0,0,0),\\
&(2,2,2;0,0,1,0), (2,2,2;0,0,0,2), (2,2,2;2,0,0,0).  
\end{align*}
In order to eliminate these representations (i.e., showing that they cancel out in the spectral sequence that calculates the cohomology of $\xi$ from that of $\xi'$) we observe the following (the calculations here are done with the program LiE):

None of the representations on the list $L(i)$ appears in the tensor product
\[
(E\otimes V)\otimes \bigoplus_{|\lambda|=i-1} \bS_\lambda E \otimes V_{[\lambda]}
\]
(and hence cannot be minimal ideal generators). So the proposed ideal generators in $L(i)$ cancel in the spectral sequence and do not appear.
\end{proof}

\begin{remark}
There is a subtlety here which is worth pointing out. The calculation in the proof above as stated is too big to be done efficiently on a computer. However, using the duality on the category of $\GL(E)\times G$ representations taking $\bS_\mu E\otimes V_\nu$ to $\bS_{\mu^\dagger} E \otimes V_\nu$ (which exchanges the exterior and symmetric powers of $\xi'$) and using the fact that we are interested in Schur functors with $\mu_1\le 2$, we see that after applying the duality we can make the further simplification of assuming that $\dim E=2$.
\end{remark}

\subsection{1 copy of $V$.} \label{sec:F41copy}
Let $Y \subset V$ be the affine cone over the highest weight orbit. Then $Y$ has dimension $16$ and codimension $10$. Note that $Y$ is a hyperplane section of the variety considered in
\S\ref{sec:1copyE6}, so the Betti table is the same as in Proposition~\ref{prop:E6betti}. In particular, the Hilbert series is
\[
\frac{1+10 T+28 T^{2}+28 T^{3}+10 T^{4}+T^{5}}{(1-T)^{16}}.
\]
Using LiE, and the Betti table in Proposition~\ref{prop:E6betti}, we can get the representation
structure of the resolution:
\begin{align*}
  \begin{array}{ll}
  \bF_0 = A &
  \bF_1 = (K \oplus V_{\omega_4})(-2)\\
  \bF_2 = (V_{\omega_1} \oplus V_{\omega_4})(-3) &
  \bF_3 = (V_{\omega_4} \oplus V_{\omega_3} \oplus V_{\omega_1})(-5)\\
  \bF_4 = (K \oplus V_{\omega_4}^{\oplus 2} \oplus V_{2\omega_4})(-6) &
  \bF_5 = (K \oplus V_{\omega_4} \oplus V_{2\omega_4})(-7) \oplus (K
  \oplus V_{\omega_4} \oplus V_{2\omega_4})(-8)\\
  \bF_6 = (K \oplus V_{\omega_4}^{\oplus 2} \oplus V_{2\omega_4})(-9) &
  \bF_7 = (V_{\omega_4} \oplus V_{\omega_3} \oplus V_{\omega_1})(-10)\\
  \bF_8 = (V_{\omega_1} \oplus V_{\omega_4})(-12) &
  \bF_9 = (K \oplus V_{\omega_4})(-13)\\
  \bF_{10} = A(-15)
\end{array}
\end{align*}


\section{Type ${\rm E}_6$.}

We assume that $K$ has characteristic $0$ in this section.

\subsection{Description of homogeneous spaces}

We label the diagram as
\[
\xymatrix @-1.2pc { & & 2 & & \\ 1 \ar@{-}[r] & 3 \ar@{-}[r] & 4
  \ar@{-}[r] \ar@{-}[u] & 5 \ar@{-}[r] & 6 }
\]
Again we start with the information on the homogeneous spaces for the group $G$ of type ${\rm E}_6$. We will follow \cite[\S 7]{carr}. Let $V$ be a $27$-dimensional irreducible representation, we can either take this to be $V_{\omega_1}$ or $V_{\omega_6}$, they are duals to each other. We will take $V = V_{\omega_1}$ for consistency with \cite{carr}. As usual, we will take $A = \Sym(E \otimes V)$ throughout, where $E$ will be a vector space whose dimension is specified in each subsection.

There is a $G$-equivariant map $\# \colon \Sym^2 V \to V^*$. Call a
subspace {\bf $\#$-isotropic} if $\#$ restricts to 0 on it. A nonzero
vector $x \in V$ is {\bf singular} if $x \# x = 0$. The subspaces of
the form $x \# V$ where $x$ is singular are 10-dimensional and are
called {\bf hyperlines}. The Levi subgroup of the
parabolic subgroup stabilizing a hyperline is of type ${\rm D}_5$, so
$x \# V$ is equipped with a nondegenerate symmetric bilinear form
$\beta$.

The ${\rm D}_5$ flag variety consists of $\beta$-isotropic flags $W_1
\subset W_2 \subset W_3 \subset W_5$ and $W'_5 \supset W_3$
($\beta$-isotropic implies $\#$-isotropic, see
\cite[(7.17)]{carr}). Exactly one of $W_5$ and $W'_5$ is contained in
a 6-dimensional $\#$-isotropic subspace $W_6$, we will label them so
that it is $W'_5$. So points of $G/B$ are given by a hyperline
$W_{10}$ together with a $\beta$-isotropic flag contained in it. For
the partial flag varieties we just forget some subspaces, but we need
to know which nodes correspond to which subspaces. In order of nodes,
they are: $W_1$, $W_6$, $W_2$, $W_3$, $W_5$, $W_{10}$. In particular,
$G/P_i$ parametrizes certain $d_i$-dimensional subspaces, and we
include some dimension data:
\[
\begin{array}{c|c|c}
  i & d_i & \dim G/P_i \\
  \hline
  1 &1 & 16\\ 
  2 & 6 & 21\\
  3 & 2 & 25\\
  4 & 3 & 29\\
  5 & 5 & 25\\
  6 & 10 & 16
\end{array}
\]

We set $W_4 = W_5 \cap W'_5$. Let $\cL(a_1, \dots, a_6)$ denote the line bundle whose sections form the dual of the irreducible with highest $a_1\omega_1 + \cdots + a_6\omega_6$. Each $W_i$ gives rise to a tautological subbundle $\cR_i$ of the trivial bundle $V \times G/B$. We can also define $\cR'_5$. Also, $\cR_{10}$ has the structure of an orthogonal bundle, i.e., we have a map $\Sym^2 \cR_{10} \to \cM$ for some line bundle $\cM$. Observe that $\bigwedge^{10} \cR_{10} \cong \cM^{\otimes 5}$, and the only submodule of $\bigwedge^{10} V$ whose highest weight is a multiple of $5$ is $V_{5\omega_6}$. So we get $\cM = \cL(0,0,0,0,0,-1)$. We can also define $\cR_{10-i} = \cR_i^\perp$ by taking $\beta$-orthogonal complements.

Note that $\bigwedge^2 V = V_{\omega_3}$, $\bigwedge^3 V = V_{\omega_4}$, $\bigwedge^4 V = V_{\omega_2 + \omega_5}$, and $\bigwedge^5 V = V_{2\omega_5} \oplus V_{2\omega_2 + \omega_6}$. We see that the duals of the determinants of $\cR_1, \cR_2, \cR_3, \cR_4$ give rise to embeddings of $G/B$ into $\bP(V_{\omega_3})$, $\bP(V_{\omega_4})$, and $\bP(V_{\omega_2 + \omega_5})$. Also, $(\det \cR_5)^*$ gives the embedding into $\bP(V_{2\omega_5})$.  

\begin{proposition}
Under the above identifications we have
\begin{align*}
  \cR_1 = \cL(-1,0,0,0,0,0), &\quad& \cR_2/\cR_1 =
  \cL(1,0,-1,0,0,0),\\
  \cR_3/\cR_2 = \cL(0,0,1,-1,0,0), &\quad& \cR_4/\cR_3 =
  \cL(0,-1,0,1,-1,0), \\
  \cR_5/\cR_4 = \cL(0,1,0,0,-1,0),&\quad& 
  \cR_6/\cR_5 = \cL(0,-1,0,0,1,-1),\\
  \cR_7/\cR_6 = \cL(0,1,0,-1,1,-1), &\quad& \cR_8/\cR_7 =
  \cL(0,0,-1,1,0,-1),\\
  \cR_9/\cR_8 = \cL(-1,0,1,0,0,-1), &\quad& \cR_{10}/\cR_9 =
  \cL(1,0,0,0,0,-1).
\end{align*}
\end{proposition}

\begin{proof}
Similar to proof of Proposition~\ref{prop:G2linebundles}.
\end{proof}

\subsection{10 copies of $V$.}

Now let $E$ be a $10$-dimensional vector space and $X = \hom(E,V)$. Consider the vector bundle $\hom(E, \cR_{10})$ over $G/P_6$. Let $Y \subset X$ be the image of the projection of $\hom(E,\cR_{10})$ onto $X$. 

\begin{theorem} 
The variety $Y$ is the $\GL(E) \times G$-orbit of $E^* \otimes W$ where $W$ is a hyperline, so this map is birational. The variety $Y$ is normal with rational singularities and its coordinate ring is
\[
\rH^0(G/P_6; \Sym(E \otimes \cR_{10}^*)) = \bigoplus_\lambda \bS_\lambda E \otimes \rH^0(G/P_6; \bS_\lambda \cR_{10}^*).
\]
\end{theorem}

\begin{proof}
The proof is analogous to the proof of Theorem~\ref{thm:6copiesF4}. 
\end{proof}

\begin{remark}
The codimension of $Y$ is $154$. The last term of the resolution of $Y$ contains $(\bigwedge^{10} E)^{\otimes 17} \otimes V^*_{38\omega_1} \otimes A(-170)$, so $Y$ is not Gorenstein.
\end{remark}

\begin{conjecture} 
The Tor modules $\Tor_i^A(K, K[Y])$ are Schur functors on $E$ and $V$.
\end{conjecture}

\subsection{5 copies of $V$.}

To get a spherical variety analogous to the determinantal variety, we now take $E$ to be $5$-dimensional. Let $Y \subset \hom(E, V)$ be the image of $\hom(E, \cR_5)$ over $G/P_5$. The codimension of $Y$ is $85$. For a partition $\lambda = (\lambda_1, \dots, \lambda_5)$, define
\[
[\lambda] = (\lambda_1 - \lambda_2, \lambda_4 - \lambda_5, \lambda_2 - \lambda_3, \lambda_3 - \lambda_4, \lambda_4 + \lambda_5, 0)
\]

\begin{theorem} The variety $Y$ is spherical and normal with rational singularities. Its coordinate ring decomposes as follows
\[
K[Y]=\bigoplus_\lambda \bS_\lambda E \otimes \rH^0(G/P_5; \bS_\lambda \cR_5^*) = \bigoplus_\lambda \bS_\lambda E \otimes V^*_{[\lambda]}.
\]
\end{theorem}

\begin{proof}
The proof is analogous to the proof of Theorem~\ref{thm:6copiesF4}.
\end{proof}

\begin{remark}
Here the calculation of the exterior powers of $V$ reveals that the defining ideal of $Y$ cannot be generated by quadrics. Indeed $\bigwedge^2 V$ is irreducible, so the only equations in degree $2$ are given by the representation $\bS_2 E\otimes V^*_{\omega_1}$, i.e., they are the polarizations of the equations one gets for one copy of $V$. However these cannot generate the equations in degree $5$ with the $\GL(E)$ component $\bigwedge^5 E$. However, instead of using $\cR_5$, we could have defined our variety using the bundle $\cR'_5$. This is similar to the situation that occurs in \S\ref{sec:typeDn}. So it is plausible that the defining ideal of the union of the two varieties we get in this way is generated by quadrics. 
\end{remark}

\begin{conjecture} 
Let $\dim E'=4$ with $E'$ a quotient of $E$. The defining ideal of $Y \cap \hom(E',V)$ is generated by quadrics. The defining equations are given by $\bS_2 E' \otimes V^*_{\omega_1}$.
\end{conjecture}

\subsection{3 copies of $V$.}

We now take $E$ to be 3-dimensional and set $X = \hom(E,V)$. Consider the vector bundle $\hom(E, \cR_{3})$ over $G/P_4$. Let $Y \subset X$ be the image of the projection of $\hom(E,\cR_{3})$ onto $X$. 

\begin{theorem} 
The variety $Y$ is spherical and normal with rational singularities.
Its coordinate ring decomposes as
\[
K[Y]=\bigoplus_\lambda \bS_\lambda E \otimes \rH^0(G/P_4; \bS_\lambda \cR_3^*) = \bigoplus_\lambda \bS_\lambda E \otimes V^*_{[\lambda]},
\]
where $[\lambda] = (\lambda_1 - \lambda_2, 0, \lambda_2 - \lambda_3, \lambda_3,0,0)$.
\end{theorem}

\begin{proof}
Similar to the proof of Theorem~\ref{thm:6copiesF4}.
\end{proof}

\begin{theorem} \label{thm:3cope6quad} 
The defining ideal $Y$ is generated by quadrics.
\end{theorem}

\begin{proof} 
We can verify using Macaulay2 that the ideal generated by quadrics (these are the polarizations of the quadrics defined in the next section) gives the correct Hilbert function up to degree $4$. For the rest of the Hilbert function, the proof is analogous to the proof of Theorem~\ref{thm:3copf4quad}, and we get the lists
\begin{align*}
L(5) &: (2,2,1;0,0,0,0,0,1),\\
L(6) &: (2,2,2;0,0,0,0,0,0).
\end{align*}
They can be ruled out in the same exact way as in Theorem~\ref{thm:3copf4quad}. 
\end{proof}

\subsection{1 copy of $V$.} \label{sec:1copyE6}

Now let $Y$ be the affine cone over the highest weight orbit in $V$. Then the resolution in question was calculated in \cite{enrightetall} (and in the $\rE_7$ case to be discussed in \S\ref{sec:1copyE7}). In that paper the interesting interpretation of the coordinate rings of the orbit closures in terms of irreducible modules for parabolic algebras of Hermitian type was described. The resolution is then an example of a parabolic BGG resolution.

We label our variables as $x_i,y_i$ ($i=1,\dots,6$), and
$z_{ij}$ ($1 \le i < j \le 6$). Let $Z$ be a skew-symmetric matrix built from $z_{ij}$. Then
\[
P = \Pf(Z) + \sum_{i \ne j} z_{ij} x_i y_j
\]
is the unique cubic invariant on $V$, where $z_{ji} = -z_{ij}$ for $j>i$. The partial derivatives of $P$ span a copy of $V^*_{\omega_1}$ in degree 2.

\begin{proposition} \label{prop:E6betti}
The Betti table of $Y$ is
\small 
\begin{verbatim}
       0  1  2   3   4   5   6   7  8  9 10
total: 1 27 78 351 650 702 650 351 78 27  1
    0: 1  .  .   .   .   .   .   .  .  .  . 
    1: . 27 78   .   .   .   .   .  .  .  .
    2: .  .  . 351 650 351   .   .  .  .  . 
    3: .  .  .   .   . 351 650 351  .  .  .
    4: .  .  .   .   .   .   .   . 78 27  .
    5: .  .  .   .   .   .   .   .  .  .  1
\end{verbatim} 
\normalsize
\end{proposition}

\begin{proof}
From Kostant's theorem, the ideal is generated by quadrics, which we have just discussed. Using Macaulay2, we can verify that the first 4 terms of the resolution are correct. But we know it is Gorenstein (it follows from the Hilbert polynomial calculated in \cite[\S 8]{2kraskiewicz}), so the last 4 terms of the resolution are correct. Then the middle term is determined by setting the Euler characteristic equal to 0. 
\end{proof}

In particular, the Hilbert series is
\[
\frac{1+10 T+28 T^{2}+28 T^{3}+10 T^{4}+T^{5}}{(1-T)^{17}}.
\]
The terms of the resolution are
\begin{align*}
\begin{array}{lll}
\bF_0 = A & 
\bF_1 = V_{\omega_1}(-2) &
\bF_2 = V_{\omega_2}(-3) \\
\bF_3 = V_{\omega_5}(-5) &
\bF_4 = V_{\omega_1 + \omega_6}(-6) &
\bF_5 = V_{2\omega_1}(-7) \oplus V_{2\omega_6}(-8) \\
\bF_6 = V_{\omega_1 + \omega_6}(-9) &
\bF_7 = V_{\omega_3}(-10) &
\bF_8 = V_{\omega_2}(-12) \\
\bF_9 = V_{\omega_6}(-13) & 
\bF_{10} = A(-15)
\end{array}
\end{align*}


\section{Type ${\rm E}_7$.}

We assume that $K$ is a field of characteristic $0$.

\subsection{Description of homogeneous spaces}

We label the diagram as
\[
\xymatrix @-1.2pc { & & 2 & & & \\ 1 \ar@{-}[r] & 3 \ar@{-}[r] & 4
  \ar@{-}[r] \ar@{-}[u] & 5 \ar@{-}[r] & 6 \ar@{-}[r] & 7 }
\]
The minimal representation $V = V_{\omega_7}$ has dimension $56$ and all representations are self-dual. The description of flag varieties is similar to the case of ${\rm E}_6$ and we follow \cite[\S 8]{carr}. There is a unique $G$-invariant in $\Sym^4 V$, which we can turn into a map $t \colon \Sym^3 V \to V$. A subspace $W$ is an {\bf inner ideal} if $t(W,V,W) \subset W$. The Levi subgroup of a parabolic subgroup stabilizing a $12$-dimensional inner ideal is of type ${\rm D}_6$, so such subspaces are equipped with a nondegenerate bilinear form $\beta$. The points of the flag variety $G/B$ are described by a $12$-dimensional inner ideal $W_{12}$ along with a maximal $\beta$-isotropic flag $W_1 \subset W_2 \subset W_3 \subset W_4 \subset W_6$ and $W_4 \subset W'_6$ inside of $W_{12}$. Each $W_i$ is also an inner ideal, and exactly one of $W_6$ and $W'_6$ is contained in a $7$-dimensional inner ideal $W_7$, by convention this one will be $W'_6$. The partial flag varieties are obtained by forgetting some of the subspaces, the ordering of the nodes corresponds to the following ordering of the subspaces: $W_{12}, W_7, W_6, W_4, W_3, W_2, W_1$. Also we define $W_5 = W_6 \cap W'_6$.

This gives us tautological subbundles $\cR_1, \dots, \cR_{12}$ and
$\cR'_6$. Also, we have a map $\Sym^2 \cR_{12} \to \cM$ where $\cM =
\cL(-1,0,0,0,0,0,0)$. Then following the calculations from before, we can deduce that 
\begin{align*}
  \cR_1 = \cL(0,0,0,0,0,0,-1), &\quad& \cR_2/\cR_1 =
  \cL(0,0,0,0,0,-1,1),\\
  \cR_3/\cR_2 = \cL(0,0,0,0,-1,1,0), &\quad& \cR_4/\cR_3 =
  \cL(0,0,0,-1,1,0,0),\\
  \cR_5/\cR_4 = \cL(0,-1,-1,1,0,0,0), &\quad& \cR_6/\cR_5 =
  \cL(0,1,-1,0,0,0,0),\\
  \cR'_6/\cR_5 = \cL(-1,-1,1,0,0,0,0), &\quad& \cR_7/\cR_6 =
  \cL(-1,-1,1,0,0,0,0),\\
  \cR_7/\cR'_6 = \cL(0,1,-1,0,0,0,0), &\quad& \cR_8/\cR_7 =
  \cL(-1,1,1,-1,0,0,0),\\
  \cR_{9}/\cR_8 = \cL(-1,0,0,1,-1,0,0), &\quad& \cR_{10}/\cR_9 =
  \cL(-1,0,0,0,1,-1,0),\\
  \cR_{11}/\cR_{10} = \cL(-1,0,0,0,0,1,-1), &\quad& \cR_{12}/\cR_{11}
  = \cL(-1,0,0,0,0,0,1).
\end{align*}

The dimension of $G/P_1$ is 33. 

\subsection{12 copies of $V$.}

Let $E$ be a 12-dimensional vector space and let $Y \subset \hom(E,V)$ be the image under the projection $\pi$ of $\hom(E,\cR_{12})$ over $G/P_1$. Let $A = \Sym(E \otimes V)$. As in the other cases, we can prove the following.

\begin{theorem} 
The variety $Y$ is the $\GL(E) \times G$-orbit of $E^* \otimes W$ where $W$ is an inner ideal, so $\pi$ is birational. The variety $Y$ is normal with rational singularities and its coordinate ring is
\[
\rH^0(G/P_1; \Sym(E \otimes \cR_{12}^*)) = \bigoplus_\lambda \bS_\lambda E \otimes \rH^0(G/P_1; \bS_\lambda \cR_{12}^*).
\]
\end{theorem}

The last term in the resolution of $Y$ is $(\bigwedge^{12} E)^{\otimes 44} \otimes
V_{55\omega_1} \otimes A(-528)$, so $Y$ is not Gorenstein. The
codimension of $Y$ is $495$.

\begin{conjecture} 
The Tor modules $\Tor_i^A(K, K[Y])$ are Schur functors on $E$ and $V$.
\end{conjecture}

\subsection{6 copies of $V$.} \label{sec:6copiesE7}

Let $E$ be a 6-dimensional vector space and let $Y \subset \hom(E,V)$ be the image of $\hom(E,\cR_6)$ over $G/P_3$. Define
\[ 
[\lambda] = (0, \lambda_5 - \lambda_6, \lambda_5 + \lambda_6, \lambda_4 - \lambda_5, \lambda_3 - \lambda_4, \lambda_2 - \lambda_3, \lambda_1 - \lambda_2).
\]
As before, we get the following result.

\begin{theorem} 
The variety $Y$ is spherical and normal with rational singularities. Its coordinate ring decomposes as follows
\[
K[Y]=\bigoplus_\lambda \bS_\lambda E \otimes \rH^0(G/P_3; \bS_\lambda \cR_6^*) = \bigoplus_\lambda \bS_\lambda E \otimes V^*_{[\lambda]}.
\]
\end{theorem}

\subsection{1 copy of $V$.} \label{sec:1copyE7}

As already remarked in \S\ref{sec:1copyE6}, the resolution of the affine cone over the highest weight orbit in $V$ was calculated in \cite{enrightetall}. In that paper the interesting interpretation of the coordinate rings of the orbit closures in terms of irreducible modules for parabolic algebras of Hermitian type was described.


\section{Type ${\rm E}_8$.}

We assume that $K$ is a field of characteristic $0$ in this section.

\subsection{Description of homogeneous spaces}

We label the Dynkin diagram of type $\rE_8$ as
\[
\xymatrix @-1.2pc { & & 2 & & & \\ 1 \ar@{-}[r] & 3 \ar@{-}[r] & 4
  \ar@{-}[r] \ar@{-}[u] & 5 \ar@{-}[r] & 6 \ar@{-}[r] & 7 \ar@{-}[r] &  8 }
\]
The projection $G/B \to G/P_1$ is a relative orthogonal flag variety of type $\rD_7$, so we have an orthogonal bundle $\cR_{14}$ on $G/B$ which is the pullback of a bundle on $G/P_1$. As before, we can define $\cR_i$ for $i=1,\dots,13$ and $\cR'_7$. With notation following the previous cases, we have the following calculations:
\begin{align*}
  \cR_1 = \cL(0,0,0,0,0,0,0,-1), &\quad& \cR_2/\cR_1 =
  \cL(0,0,0,0,0,0,-1,1),\\
  \cR_3/\cR_2 = \cL(0,0,0,0,0,-1,1,0), &\quad& \cR_4/\cR_3 =
  \cL(0,0,0,0,-1,1,0,0),\\
  \cR_5/\cR_4 = \cL(0,0,0,-1,1,0,0,0), &\quad& \cR_6/\cR_5 =
  \cL(0,-1,-1,1,0,0,0,0),\\
  \cR_7/\cR_6 = \cL(0,1,-1,0,0,0,0,0) &\quad& \cR_7'/\cR_6 =
  \cL(-1,-1,1,0,0,0,0,0) \\
  \cR_8/\cR_7 = \cL(-1,-1,1,0,0,0,0,0) &\quad& \cR_8/\cR'_7 =
  \cL(0,1,-1,0,0,0,0,0) \\
  \cR_9/\cR_8 = \cL(-1,1,1,-1,0,0,0,0) &\quad& \cR_{10}/\cR_9 =
  \cL(-1,0,0,1,-1,0,0,0) \\
  \cR_{11}/\cR_{10} = \cL(-1,0,0,0,1,-1,0,0) &\quad& \cR_{12}/\cR_{11} =
  \cL(-1,0,0,0,0,1,-1,0) \\
  \cR_{13}/\cR_{12} = \cL(-1,0,0,0,0,0,1,-1) &\quad& \cR_{14}/\cR_{13}
  = \cL(-1,0,0,0,0,0,0,1)
\end{align*}

This is similar to the previous two cases. Let $V = V_{\omega_8}$ be
the adjoint representation. 

The dimension of $G/P_1$ is 78. 

\subsection{14 copies of $V$.}

Let $E$ be 14-dimensional and let $Y \subset \hom(E,V)$ be the image of $\hom(E,\cR_{14})$ over $G/P_1$. Set $A = \Sym(E \otimes V)$. The codimension of $Y$ is 3198. Then the last term of the resolution of
$Y$ contains $(\bigwedge^{14} E)^{\otimes 234} \otimes V_{1615\omega_1} \otimes A(-3276)$ so $Y$ is not Gorenstein. But we can show that $Y$ is normal with rational singularities (similar to ${\rm F}_4$ case). 

\begin{conjecture} 
The Tor modules $\Tor_i^A(K, K[Y])$ are Schur functors on $E$ and $V$.
\end{conjecture}

\subsection{7 copies of $V$.}

For a 7-dimensional vector space $E$, the situation is similar to \S\ref{sec:6copiesE7}. In this situation, we define
\[
[\lambda] = (0, \lambda_6 - \lambda_7, \lambda_6 + \lambda_7, \lambda_5 - \lambda_6, \lambda_4 - \lambda_5, \lambda_3 - \lambda_4, \lambda_2 - \lambda_3, \lambda_1 - \lambda_2).
\]

\subsection{1 copy of $V$.}

The affine cone over $G/P_8$ has dimension 58 and hence codimension
190 inside of $V$. Just counting representations by hand, the
resolution starts like this:
\[
(V_{\omega_8} \oplus V_{\omega_2} \oplus V_{\omega_1}) \otimes A(-3)
\to (K \oplus V_{\omega_1}) \otimes A(-2) \to A
\]


\end{document}